\documentclass[onecolumn, reqno]{amsart}
\usepackage{amssymb, amsmath, mathrsfs, subfigure, url}
\usepackage[dvipsnames, usenames]{color}
\usepackage[colorlinks=true,
        raiselinks=true,
        linkcolor=MidnightBlue,
        citecolor=Mahogany,
        urlcolor=ForestGreen,
        pdfauthor={P. Hokayem, E. Cinquemani, D. Chatterjee, F. Ramponi,  J. Lygeros},
        pdftitle={Stochastic receding horizon control with output feedback and bounded control inputs},
        pdfkeywords={Predictive control, output feedback, constrained control, state estimation, stochastic control},
        pdfsubject={Technical Report},
        plainpages=false]{hyperref}
\usepackage{graphicx, mathtools}
\usepackage{algorithmic}
\usepackage{algorithm,enumitem}
\usepackage[square, sort&compress, authoryear]{natbib}

\allowdisplaybreaks
\DeclareMathOperator*{\minimize}{minimize}

\def\NN{\mathbb{N}_+}
\def\RR{\mathbb{R}}

\def\Nz{\mathbb{N}}
\def\R{\mathbb{R}}

\def\EE{\mathbb{E}}
\def\PP{\mathbb{P}}
\def\U{\mathbb{U}}

\def\x0{x_0}

\newcommand{\state}{x}           % state
\newcommand{\State}{X}           % augmented state
\newcommand{\inp}{u}             % input
\newcommand{\Inp}{U}             % augmented input
\newcommand{\out}{y}             % output
\newcommand{\Out}{Y}             % augmented output
\newcommand{\error}{e}           % estimation error
\newcommand{\Error}{E}           % augmented estimation error vector
\newcommand{\wnoise}{w}          % process noise
\newcommand{\vnoise}{v}          % measurement noise
\newcommand{\Wnoise}{W}          % augmented process noise
\newcommand{\Vnoise}{V}          % augmented measurement noise
\newcommand{\Y}{\mathcal Y}      % output filtration

\newcommand{\est}{{\hat\state}}			% estimate of the general process
\newcommand{\Kal}{K}             % Kalman gain
\newcommand{\Reach}{\mathfrak{R}}% reachability matrix
\newcommand{\0}{0}               % zeros matrix
\newcommand{\I}{I}               % identity matrix
\newcommand{\M}{\mathcal M}      % some generic weighting matrix
\newcommand{\F}{\mathcal F}      % some generic weighting matrix
\newcommand{\Wu}{R}				% input weight
\newcommand{\Wx}{Q}				% state weight
\newcommand{\Wua}{\mathcal R}    % augmented R
\newcommand{\Wxa}{\mathcal Q}    % augmented Q
\newcommand{\A}{A}   % state matrix in the general system
\newcommand{\B}{B}   % input matrix in the general system
\newcommand{\C}{C}   % output matrix in the general system
\newcommand{\Aa}{\mathcal A}     % augmented A
\newcommand{\Ba}{\mathcal B}     % augmented B
\newcommand{\Ca}{\mathcal C}     % augmented C
\newcommand{\Da}{\mathcal D}     % augmented D

\newcommand{\LambdaPhi}{\Lambda^{\varphi}}
\newcommand{\LambdaPhiE}{\Lambda^{\varphi e}}
\newcommand{\LambdaPhiX}{\Lambda^{\varphi x}}
\newcommand{\LambdaWPhi}{\Lambda^{w\varphi}}
\newcommand{\LambdaPhiPhi}{\Lambda^{\varphi\varphi}}

\newcommand{\Pmat}{P}		% conditional variance of the error in the general system

\newcommand{\sigmin}{\sigma_{\min}}     % minimum singular value
\newcommand{\sigmax}{\sigma_{\max}}     % maximum singular value
\newcommand{\ball}{{\mathcal B}}        % a given ball
\DeclareMathOperator{\sat}{sat}
\newcommand{\var}[1]{\Sigma_{#1}}
	\newcommand{\Pmatone}{M}
	\newcommand{\Qmatone}{I}
	\newcommand{\estone}{\est^{(1)}}
	\newcommand{\esttwo}{\est^{(2)}}
	\newcommand{\inpone}{\inp^{(1)}}
	
	\newcommand{\olgain}{\boldsymbol{\eta}}
	\newcommand{\clgain}{\boldsymbol{\Theta}}
	\newcommand{\validfrom}{{T}'}
	\newcommand{\kjumpvalidfrom}{{\ceil{\validfrom/\kappa}}}
\newcommand{\ii}{i}

\newcommand{\thehorror}{\zeta}
\newcommand{\partialhorror}{\rho}

\newcommand{\fa}{\forall\,}

\newcommand{\Let}{\coloneqq}
\newcommand{\teL}{\eqqcolon}

\newcommand{\sigalg}{\mathfrak F}

\renewcommand{\le}{\leqslant}
\renewcommand{\leq}{\leqslant}
\renewcommand{\ge}{\geqslant}
\renewcommand{\geq}{\geqslant}

\newcommand{\tr}[1]{\mathsf{tr}\!\left(#1\right)}
\newcommand{\nn}{\nonumber}

\newcommand{\transp}{^\mathsf{T}}

\newcommand{\smat}[1]{\left[\begin{matrix} #1 \end{matrix}\right]}
\newcommand{\abs}[1]{\left\lvert{#1}\right\rvert}
\newcommand{\ceil}[1]{\left\lceil{#1}\right\rceil}

\newcommand{\AssmpEnd}{\hspace{\stretch{1}}{$\diamondsuit$}}

\newcommand{\norm}[1]{\left\lVert{#1}\right\rVert}

\newcommand{\eps}{\varepsilon}

\newtheorem{theorem}{Theorem}
\newtheorem{corollary}[theorem]{Corollary}
\newtheorem{lemma}[theorem]{Lemma}
\newtheorem{proposition}[theorem]{Proposition}
\newtheorem{scholium}[theorem]{Scholium}
\newtheorem{fact}[theorem]{Fact}
\theoremstyle{remark}

\newtheorem{assumption}[theorem]{Assumption}

\numberwithin{equation}{section}

\begin{document}

\title[Stochastic Receding Horizon Control\ldots]{Stochastic receding horizon control with output feedback and bounded control inputs}

\thanks{This research was partially supported by the Swiss National Science Foundation under grant 200021-122072 and by the European Commission under the project Feednetback FP7-ICT-223866 (www.feednetback.eu). This article was not presented at any IFAC meeting.}
\thanks{P.\ Hokayem, D.\ Chatterjee, F. Ramponi, and J.\ Lygeros are with the Automatic Control Laboratory, ETH Z\"urich, Physikstrasse 3, 8092 Z\"urich, Switzerland. E.\ Cinquemani is with the INRIA Grenoble-Rh\^one-Alpes. Corresponding author P.\ Hokayem Tel.\ +41 44 632 0403. Fax:\ +41 44 632 1211.}
\thanks{Email: \{hokayem,chatterjee,ramponif,lygeros\}@control.ee.ethz.ch,\\ Eugenio.Cinquemani@inria.fr}

\author[P.~Hokayem]{Peter Hokayem}
\author[E.~Cinquemani]{Eugenio Cinquemani}
\author[D.~Chatterjee]{Debasish Chatterjee}
\author[F.~Ramponi]{Federico Ramponi}
\author[J.~Lygeros]{John Lygeros}


\begin{abstract}
	We provide a solution to the problem of receding horizon control for stochastic discrete-time systems with bounded control inputs and imperfect state measurements. For a suitable choice of control policies, we show that the finite-horizon optimization problem to be solved on-line is convex and successively feasible.  Due to the inherent nonlinearity of the feedback loop, a slight extension of the Kalman filter is exploited to estimate the state optimally in mean-square sense. We show that the receding horizon implementation of the resulting control policies renders the state of the overall system mean-square bounded under mild assumptions. Finally, we discuss how some of the quantities required by the finite-horizon optimization problem can be computed off-line, reducing the on-line computation, and present some numerical examples.
\end{abstract}

\maketitle

%--------------------------------------------------------------------------------------
\section{Introduction}
%--------------------------------------------------------------------------------------
A considerable amount of research effort has been devoted to deterministic receding horizon control, see, for example, \citep{MayneRawlingsRaoScokaert-00,BemporadMorari-99,ref:lazarbemporad07,ref:maciejowskibk,ref:JoeQinBadgwell-03} and references therein. Currently we have readily available conclusive proofs of successive feasibility and stability of receding horizon control laws in the noise-free deterministic setting. These techniques can be readily extended to the robust case, i.e., whenever there is an additive noise of bounded nature entering the system. The counterpart for stochastic systems subject to process noise and imperfect state measurements and bounded control inputs, however, is still missing. The principal obstacle is posed by the fact that it may not be possible to determine an a priori bound on the support of the noise, e.g., whenever the noise is additive and Gaussian. This extra ingredient complicates both the stability and feasibility proofs manifold: the noise eventually drives the state outside any bounded set no matter how large the latter is taken to be, and employing any standard linear state feedback means that any hard bounds on the control inputs will eventually be violated.

In this article we propose a solution to the general receding horizon control problem for linear systems with noisy process dynamics, imperfect state information, and bounded control inputs. Both the process and measurement noise sequences are assumed to enter the system in an additive fashion, and we require that the designed control policies satisfy hard bounds. Periodically at times $t=0,N_c,2N_c,\cdots$, where $N_c$ is the control horizon, a certain finite-horizon optimal control problem is solved for a prediction horizon $N$. The underlying cost is the standard expectation of the sum of cost-per-state functions that are quadratic in the state and control inputs \citep{ref:Ber-05}. We also impose some variance-like bounds on the predicted future states and inputs---this is one possible way to impose soft state-constraints that are in spirit similar to integrated chance-constraints, e.g., in \citep{ref:Han-83, ref:Han-06}.

There are several key challenges that are inherent to our setup. First, since state-information is incomplete and corrupt, the need for a filter to estimate the state from the corrupt state measurements naturally presents itself. Second, it is clear that in the presence of unbounded (e.g., Gaussian) noise, it is not in general possible to ensure any bound on the control values with linear state feedback alone (assuming complete state information is available)---the additive nature of the noise ensures that the states exit from any fixed bounded set at some time almost surely. We see at once that nonlinear feedback is essential, and this issue is further complicated by the fact that only incomplete and imperfect state-information is available. Third, it is unclear whether applicatino of the control policies stabilizes the system in any reasonable sense.

In the backdrop of these challenges, let us mention the main contributions of this article. We
\begin{enumerate}[label=(\roman*), leftmargin=*, align=right, widest=ii]
  \item provide a \emph{tractable} method for designing, in a receding horizon fashion, bounded causal nonlinear feedback policies, and
  \item guarantee that applying the designed control policies in a receding horizon fashion renders the state of the closed-loop system \emph{mean-square bounded for any initial state statistics}.
\end{enumerate}
We elaborate on the two points above:
\begin{enumerate}[label=(\roman*), leftmargin=*, align=right, widest=ii]
	\item \emph{Tractability:} Given a subclass of causal policies, we demonstrate that the underlying optimal control problem translates to a convex optimization problem to be solved every $N_c$ time steps, that this optimization problem is feasible for any statistics of the initial state, and equivalent to a quadratic program. The subclass of polices we employ is comprised of open-loop terms and nonlinear feedback from the innovations. Under the assumption that the process and measurement noise is Gaussian, (even though the bounded inputs requirement makes the problem inherently nonlinear and it may appear that standard Kalman filtering may not apply,) it turns out that Kalman filtering techniques can indeed be utilized. We provide a low-complexity algorithm (essentially similar to standard Kalman filtering) for updating this conditional density, and report tractable solutions for the {\em off-line} computation of the time-dependent optimization parameters.
	\item \emph{Stability:} It is well known that for noise-free discrete-time linear systems it is not possible to globally asymptotically stabilize the system, if the usual matrix $A$ has unstable eigenvalues, see, for example, \citep{ref:YangSontagSussmann-97} and references therein. Moreover, in the presence of stochastic process noise the hope for achieving asymptotic stability is obviously not realistic. Therefore, we naturally relax the notion of stability into mean-square boundedness of the state and impose the extra conditions that the system matrix $A$ is Lyapunov (or neutrally) stable and that the pair $(A,B)$ is stabilizable. Under such standard assumptions, we show that it is possible to augment the optimization problem with an extra {\em stability constraint}, and, consequently, that the successive application of our resulting policies renders the state of the overall system (plant and estimator) mean-square bounded.
\end{enumerate}

%--------------------------------------------------------------------------------------
\subsection*{Related Works}
%--------------------------------------------------------------------------------------
The research on stochastic receding horizon control is broadly subdivided into two parallel lines: the first  treats multiplicative noise that enters the state equations, and the second caters to additive noise. The case of multiplicative noise has been treated in some detail in \citep{Primbs-07, PrimbsSung-09, ref:CanKouWu-09}, where the authors consider also soft constraints on the state and control input. However, in this article we exclusively focus on the case of additive noise.

The approach proposed in this article stems from and generalizes the idea of affine parametrization of control policies for finite-horizon linear quadratic problems proposed in \citep{ref:ben-tal04,Ben-TalBoydNemirovski-06}, utilized within the robust MPC framework in \citep{Ben-TalBoydNemirovski-06,Loefberg-03,GoulartKerriganMaciejowski-06} for full state feedback, and in \citep{vanHessemBosgra-03} for output feedback with Gaussian state and measurement noise inputs. More recently, this affine approximation was utilized in \citep{SkafBoyd-09} for both the robust deterministic and the stochastic setups in the absence of control bounds, and optimality of the affine policies in the scalar deterministic case was reported in \citep{BerIanPar-09}. In \citep{BertsimasBrown-07} the authors reformulate the stochastic programming problem as a deterministic one with bounded noise and solve a robust optimization problem over a finite horizon, followed by estimating the performance when the noise can take unbounded values, i.e., when the noise is unbounded, but takes high values with low probability. Similar approach was utilized in \citep{OldewurtelJonesMorari-08} as well. There are also other approaches, e.g., those employing randomized algorithms as in \citep{batinaPhDthesis,BlackmoreWilliams-07,MaciejowskiLecchiniLygeros-05}. Results on obtaining lower bounds on the value functions of the stochastic optimization problem have been recently reported in \citep{WangBoyd-09}.

\subsection*{Organization}
%--------------------------------------------------------------------------------------
The remainder of this article is organized as follows. In Section \ref{sec:problem} we formulate the receding horizon control problem with all the underlying assumptions, the construction of the estimator, and the main optimization problems to be solved. In Section \ref{sec:mainresult} we provide the main results pertaining to tractability of the optimization problems and mean-square boundedness of the closed-loop system. We comment on the obtained results and provide some extensions in Section \ref{sec:discussion}. We provide all the needed preliminary results, derivations, and proofs in Section \ref{sec:proofs}, and some numerical examples in Section \ref{sec:examples}. Finally, we conclude in Section \ref{sec:conclusions}.

%--------------------------------------------------------------------------------------
\subsection*{Notation}
%--------------------------------------------------------------------------------------
Let $(\Omega,\mathfrak F, \mathbb P)$ be a general probability space, we denote the conditional expectation given the sub-$\sigma$ algebra $\mathfrak F'$ of $\mathfrak F$ as $\EE_{\mathfrak F'}[.]$. For any random vector $s$ we let $\Sigma_{s}\Let\EE[ss\transp ]$. Hereafter we let $\NN \Let \{1,2,\ldots\}$ and $\Nz \Let \NN \cup \{0\}$ be the sets of positive and nonnegative integers, respectively. We let  $\tr{\cdot}$ denote the trace of a square matrix, $\norm{\cdot}_p$ denote the standard $\ell_p$ norm, and simply $\norm{.}$ denote the $\ell_2$-norm. In a Euclidean space we denote by $\ball_r$ the closed Euclidean ball or radius $r$ centered at the origin.  For any two matrices $A$ and $B$ of compatible dimensions, we denote by $\Reach_k(A,B)$ the $k$-th step reachability matrix $\Reach_k(A,B)\Let\smat{A^{k-1}B &  \cdots & AB & B}$. For any matrix $M$, we let $\sigmin$ and $\sigmax$ be its minimal and maximal singular values, respectively. We let $(M)_{i_1:i_2}$ denote the sub-matrix obtained by selecting the rows $i_1$ through $i_2$ of $M$, and let $(M)_i$ denote its $i$-th row.

%--------------------------------------------------------------------------------------
\section{Problem Setup}\label{sec:problem}
%--------------------------------------------------------------------------------------
We are given a certain plant with discrete-time dynamics, process noise $\wnoise_t$, and imperfect state measurements  $\out_t$ tainted through noise $\vnoise_t$ (Figure \ref{fig:main}). The objective is to design a receding horizon controller which renders the overall closed-loop system mean-square bounded. We discuss the structure and the main assumptions on the dynamics of the system, the performance index to be minimized, and the construction of the optimal mean-square estimator $\est_t$ in Subsection \ref{sec:dynamics}. Then, we formalize the optimization problem to be solved in a receding horizon fashion (to generate the input policies $\inp_t$) in Subsection \ref{sec:optimization}.
\begin{figure}[h]
\centering\includegraphics[width=0.7\columnwidth]{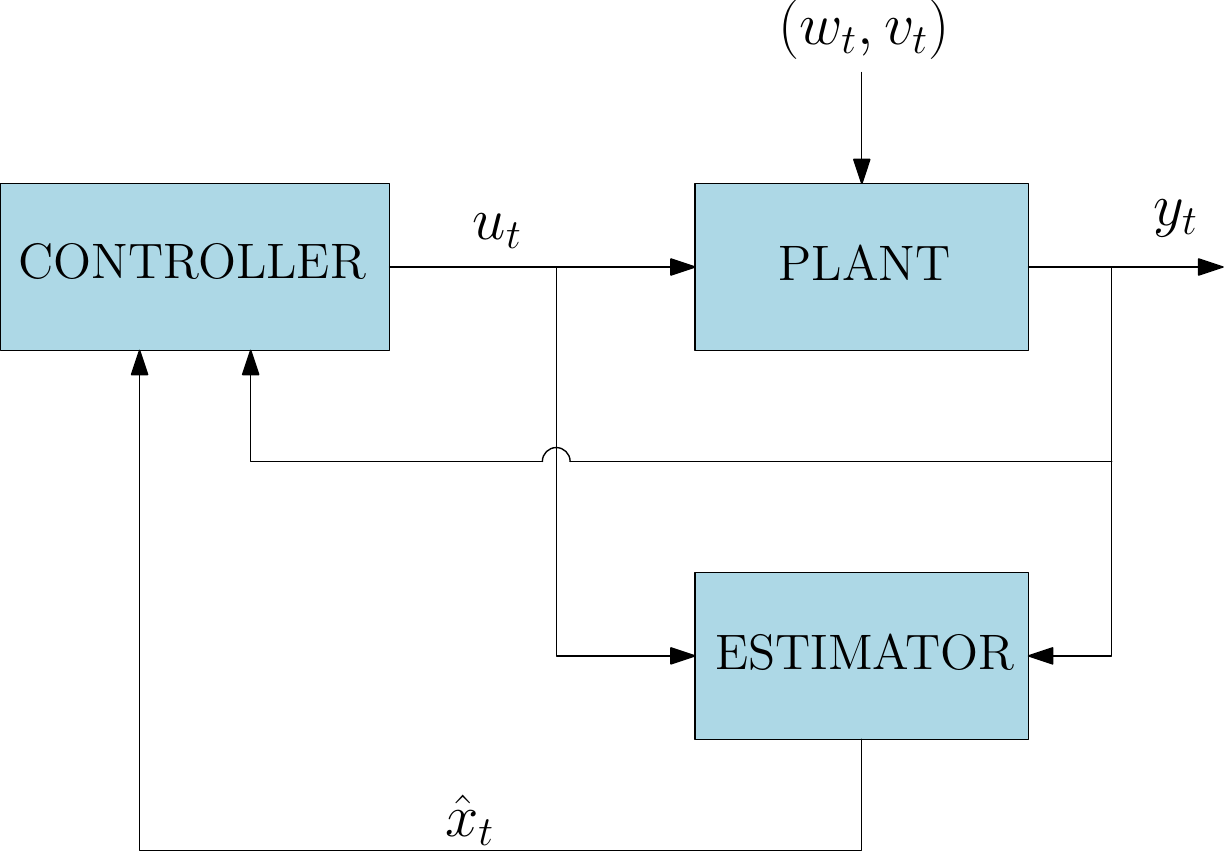}
\caption{Main setup}\label{fig:main}
\end{figure}

%--------------------------------------------------------------------------------------
\subsection{Dynamics, Performance Index and Estimator}\label{sec:dynamics}
%--------------------------------------------------------------------------------------
Consider the following affine
discrete-time stochastic dynamical model:
\begin{subequations}
\label{eq:system}
\begin{align}
    x_{t+1} &= \A x_t + \B u_t +w_t, \label{eq:system-a}\\
    y_t &= \C x_t + v_t \label{eq:system-b}
\end{align}
\end{subequations}
where $t\in \Nz$, $x_t\in\R^n$ is the state, $u_t\in \R^m$
is the control input, $y_t\in \mathbb R^p$ is the output,  $w_t\in\R^n$ is a random process noise, $v_t\in\R^p$ is a random measurement noise, and $\A$, $\B$, and $\C$ are known matrices. We posit the following main assumption:

\begin{assumption}
\label{ass:dynamics}
\mbox{}
\begin{enumerate}[label=(\roman*), leftmargin=*, align=right, widest=iii]
    \item The system matrices in \eqref{eq:system-a} satisfy the following:\label{ass:dynamics1}%% \cite[Chapter 12]{ref:Ber-09}:
        \begin{enumerate}[label=(\alph*), leftmargin=*, align=right]
        	\item The pair $(\A, \B)$ is stabilizable.  \label{ass:dynamics1a}
        	\item $\A$ is discrete-time Lyapunov stable \cite[Chapter 12]{ref:Ber-09}, i.e., the eigenvalues $\{\lambda_i(\A)\mid i=1,\ldots, d\}$ lie in the closed unit disc, and those eigenvalues $\lambda_j(\A)$ with $\bigl|\lambda_j(\A)\bigr|=1$ have equal algebraic and geometric multiplicities. \label{ass:dynamics1c}
        \end{enumerate}
    \item The initial condition and the process and measurement noise vectors are normally distributed, i.e., \label{ass:distributions}
	$\state_0\sim\mathcal N(0,\var{\state_0})$, $\wnoise_t\sim\mathcal N (0,\var{\wnoise})$, and $\vnoise_t\sim\mathcal N (0,\var{\vnoise})$. Moreover, $\state_0$, $(\wnoise_t)_{t\in\Nz}$ and $(\vnoise_t)_{t\in\Nz}$ are mutually independent.
    \item The control inputs take values in the control constraint set:\label{ass:dynamics6}
        \begin{equation}\tag{C1}
        \label{eq:C1}
        \U \Let \bigl\{\xi \in\R^m\,\big|\,\norm{\xi}_\infty\leq U_{\max}\bigr\};
        \end{equation}
		i.e., $\inp_t\in\U$ for each $t\in\Nz$.\AssmpEnd
\end{enumerate}
\end{assumption}

		 Without any loss of generality, we also assume that $\A$ is in real Jordan canonical form.  Indeed, given a linear system described by system matrices $\bigl(\tilde \A, \tilde \B\bigr)$, there exists a coordinate transformation in the state-space that brings the pair $\bigl(\tilde \A, \tilde \B\bigr)$ to the pair $(\A, \B)$, where $\A$ is in real Jordan form \cite[p. 150]{ref:hornjohnson}.  In particular, choosing a suitable ordering of the Jordan blocks, we can ensure that the pair $(\A,  \B)$ has the form $\left(\begin{bmatrix}\A_{1} & 0\\ 0 & \A_{2}\end{bmatrix}, \begin{bmatrix}\B_1\\ \B_2\end{bmatrix}\right)$, where $\A_{1}\in\R^{n_1\times n_1}$ is Schur stable, and $\A_{2}\in\R^{n_2\times n_2}$ has its eigenvalues on the unit circle. By hypothesis \ref{ass:dynamics1}-\ref{ass:dynamics1c} of Assumption \ref{ass:dynamics}, $\A_{2}$ is therefore block-diagonal with elements on the diagonal being either $\pm 1$ or $2\times 2$ rotation matrices. As a consequence, $\A_{2}$ is orthogonal.  Moreover, since $(\A,\B)$ is stabilizable, the pair $(\A_{2}, \B_2)$ must be reachable in a number of steps $\kappa \leq n_2$ that depends on the dimension of $\A_{2}$ and the structure of $(\A_{2}, \B_2)$, i.e., $\text{rank}(\Reach_\kappa(A_2,B_2))=n_2$.  Summing up, we can start by considering that the state equation  \eqref{eq:system-a} has the form
		\begin{equation}\label{eqn:Jordan}
			\begin{bmatrix} \state^{(1)}_{t+1}\\ -- \\ \state^{(2)}_{t+1}\end{bmatrix}
			= \begin{bmatrix} \A_{1} \state^{(1)}_t\\ -- \\ \A_{2} \state^{(2)}_t\end{bmatrix}
			+ \begin{bmatrix} \B_1\\ -- \\ \B_2\end{bmatrix} u_t
			+ \begin{bmatrix} \wnoise^{(1)}_t\\ -- \\ \wnoise^{(2)}_t\end{bmatrix},
		\end{equation}
		where $\A_{1}$ is Schur stable, $\A_{2}$ is orthogonal, and there exists a nonnegative integer $\kappa\leq n_2$ such that the subsystem $(\A_{2}, \B_2)$ is reachable in $\kappa$ steps. This integer $\kappa$ is fixed throughout the rest of the article.

Fix a prediction horizon $N\in\NN$, and let us consider, at any time $t$, the cost $V_t \Let V(t, \Y_t, u)$  defined by:
\begin{align}
    V_t &= \EE_{\Y_t}\Bigg[ \sum\limits_{k=0}^{N-1}\!\bigl(x_{t+k}\transp \Wx_k x_{t+k}+u_{t+k}\transp \Wu_k u_{t+k}\bigr)+x_{t+N}\transp \Wx_Nx_{t+N}\Bigg],\label{eq:totalcost}
\end{align}
where $\Y_t=\{y_0,y_1,\ldots, y_t\}$ is the set of outputs up to time $t$, (or more precisely the $\sigma$-algebra generated by the set of outputs,) and $\Wx_k=\Wx\transp _k>0$, $\Wx_N = \Wx_N\transp  > 0$, and $\Wu_k=\Wu\transp _k>0$ are some given symmetric matrices of
appropriate dimension, $k = 0,\ldots, N-1$.  At each time instant $t$, we are interested in minimizing \eqref{eq:totalcost} over the class of causal output feedback strategies $\mathcal G$ defined as:
\begin{equation}
\label{e:policies}
\smat{\inp_t \\ \inp_{t+1}\\ \vdots \\ \inp_{t+N-1}}\!\! = \!\!\left[ \begin{array}{l} g_t(\Y_t) \\
g_{t+1}(\Y_{t+1}) \\ \vdots \\ g_{t+N-1}(\Y_{t+N-1})\end{array}\right],
\end{equation}
for some measurable functions $\mathbf{g}_t\Let\{g_t,$ $\cdots, g_{t+N-1}\}$ which satisfy the hard constraint \eqref{eq:C1} on the inputs.

The cost $V_t$ in \eqref{eq:totalcost} is a conditional expectation given the observations up through time $t$, and as such requires the conditional density $f(\state_t|\Y_t)$ of the state given the previous and current measurements. Our choice of causal control policies in~\eqref{e:policies} motivates us to rewrite the cost $V_t$ in~\eqref{eq:totalcost} as:
\begin{align}
    V_t &=\EE_{\Y_t}\!\Bigg[\EE_{\{\Y_t,\state_t\}}\!\Bigg[\sum\limits_{k=0}^{N-1}\big(\state_{t+k}\transp \Wx_k \state_{t+k}+\inp_{t+k}\transp \Wu_k \inp_{t+k}\big) + \state_{t+N}\transp \Wx_N\state_{t+N}\Bigg]\Bigg]. \label{e:conditionalversionofcost}
\end{align}

The propagation of $f(\state_t|\Y_t)$ in time is not a trivial task in general. In what follows we shall report an iterative method for the computation of $f(\state_t|\Y_t)$, whenever the control input is a measurable deterministic feedback
from the past and present outputs. For $t,s\in \mathbb N_0$, define $\hat{\state}_{t|s}=\EE_{\Y_s} [\state_t]$ and $P_{t|s}=\EE_{\Y_s}[(\state_t-\hat{\state}_{t|s})(\state_t-\hat{\state}_{t|s})\transp ]$.

\begin{assumption}
\label{ass:noisematrices}
	We require that $\var{\wnoise}>0$ and $\var{\vnoise}>0$.\AssmpEnd
\end{assumption}
The following result is a slight extension of the usual Kalman Filter for which a proof can be found in Appendix A. An alternative proof can also be found in \cite[p.102]{ref:KumVar-86}.
\begin{proposition}\label{prop:1}
Under Assumption \ref{ass:dynamics}-\ref{ass:distributions} and Assumption \ref{ass:noisematrices}, $f(\state_t|\Y_t)$ and $f(\state_{t+1}|\Y_t)$
are the probability densities of Gaussian distributions
$\mathcal N(\hat{\state}_{t|t},P_{t|t})$ and $\mathcal N(\hat{\state}_{t+1|t},P_{t+1|t})$,
respectively, with $P_{t|t}>0$ and $P_{t+1|t}>0$. For
$t=0,1,2,\ldots,$ their conditional means and covariances can be
computed iteratively as follows:
{\small \begin{align}
\hat{\state}_{t+1|t+1}&=\hat{\state}_{t+1|t}\nn\\
&\quad +P_{t+1|t}\C\transp  (\C P_{t+1|t} \C\transp  +\var{\vnoise})^{-1}(\out_{t+1}-\C\hat{\state}_{t+1|t}),\label{eq:mupd}\\
P_{t+1|t+1}&=P_{t+1|t}-P_{t+1|t} \C\transp  (\C P_{t+1|t} \C\transp  +\var{\vnoise})^{-1}
\C P_{t+1|t},\label{eq:cupd}
\end{align}}
where
\begin{align}
\label{eq:mpred}
\hat{\state}_{t+1|t}&=\A\hat{\state}_{t|t}+\B g_t(\Y_t),\\
\label{eq:cpred} P_{t+1|t}&=\A P_{t|t}\A\transp +\var{\wnoise}.
\end{align}
\end{proposition}
In particular, the matrix $P_{t|t}$  together with $\hat x_{t|t}$  characterize
the conditional density $f(x_t|\Y_t)$, which is needed in the computation of the cost \eqref{e:conditionalversionofcost}.
Proposition \ref{prop:1} states that the conditional mean and covariances of
$x_t$ can be propagated by an iterative algorithm which
resembles the Kalman filter. Since the system is generally nonlinear
due to the function $g$ being nonlinear, we cannot assume that the probability
distributions in the problem are Gaussian (in fact, the a priori
distributions of $x_t$ and of $\Y_t$ are not) and the proof cannot be
developed in the standard linear framework of the Kalman filter.

Hereafter, we shall denote for notational convenience by $\hat\state_t$ the estimate $\hat\state_{t|t}$, and let
\begin{equation}\label{eqn:decomposedestimate}
    \hat\state_t=\smat{\estone_t\\ -- \\\esttwo_t },
\end{equation}
which corresponds to the Jordan decomposition in \eqref{eqn:Jordan}.

%--------------------------------------------------------------------------------------
\subsection{Optimization Problem and Control Policies} \label{sec:optimization}
%--------------------------------------------------------------------------------------
Having discussed the iterative update of the laws of $\state_t$ given the measurements $\Y_t$ in Section \ref{sec:dynamics}, we return to our optimization problem in \eqref{e:conditionalversionofcost}. We can rewrite our optimization problem as:
\begin{align}
\label{eq:problem2}
	\min_{\mathbf{g}_t} \Bigl\{V_t\ \Big| \text{ dynamics \eqref{eq:system-a}-\eqref{eq:system-b} and constraint \eqref{eq:C1}} \Bigr\}, \tag{OP1}
\end{align}
where the collection of functions $\mathbf g_t$ were defined following \eqref{e:policies}.

The explicit solution via dynamic programming to Problem  \eqref{eq:problem2} over the class of causal feedback policies $\mathcal G$, is difficult if not impossible to obtain in general \citep{ref:Ber-00Vol1,ref:Ber-07Vol2}.  One way to circumvent this difficulty is to restrict our search for feedback policies to a subclass of $\mathcal G$. This will result in a suboptimal solution to our problem, but may yield a tractable optimization problem. This is the track we pursue next.

Given a control horizon $N_c$ and a prediction horizon $N\; (\geq N_c)$, we would like to periodically solve Problem  \eqref{eq:problem2} at times $t=0,N_c,2N_c,\cdots$ over the following class of  affine-like control policies
\begin{equation}\label{eq:policy}\tag{POL}
    \inp_\ell = \eta_\ell+\sum_{i=t}^{\ell}\theta_{\ell,i}\varphi_i(\out_i-\hat \out_i) \quad \text{for all }\ell=t,\cdots,t+N-1,
\end{equation}
where $\hat\out_i=\C\hat\state_i$, and for any vector $z\in\RR^p$, $\varphi_i(z)\Let\bigl[\varphi_{i,1}(z_{1}), \ldots,$ $\varphi_{i,p}(z_{p})\bigr]\transp$, where $z_{j}$ is the $j$-th element of the vector $z$
and $\varphi_{i,j}:\R\to\R$ is any function with $\sup\limits_{s\in
\mathbb{R}}|\varphi_{i,j}(s)|\leq \varphi_{\max} < \infty$.
The feedback gains $\theta_{\ell,i}\in\R^{m\times p}$ and the affine terms
$\eta_\ell\in\R^m$ are the decision variables. The value of $\inp_\ell$ in \eqref{eq:policy} depends on the values of the measured outputs from the beginning of the prediction horizon at time $t$ up to time $\ell$ only, which requires a \emph{finite amount of memory}. Note that
 we have chosen to saturate the measurements we obtain
from the vectors $(\out_i-\hat \out_i)$ before inserting them into the control
policy. This way we do not assume that either the process noise or the measurement noise distributions are
defined over a compact domain, which is a crucial assumption in robust MPC approaches \citep{MayneRawlingsRaoScokaert-00}. Moreover, the choice of element-wise saturation functions $\varphi_i(\cdot)$ is left open. As such, we can
accommodate standard saturation, piecewise linear, and sigmoidal
functions, to name a few.

Finally, we collect all the ingredients of the stochastic receding horizon control problem in Algorithm \ref{algo:general} below.
\begin{algorithm}[hhh]
\caption{Basic Stochastic Receding Horizon Algorithm}\label{algo:general}
\begin{algorithmic}[1]
\REQUIRE density $f(\state_0|\Y_{-1})= \mathcal N(0,\var{\state_0})$
\STATE set $t=0$, $\hat \state_{0|-1}=0$, and $P_{0|-1}=\var{\state_0}$
\LOOP

    \FOR{$i=0$ to $N_c-1$}

	\STATE measure $y_{t+i}$

    \STATE update $\hat\state_{t+i}(=\hat\state_{t+i|t+i})$ and $P_{t+i|t+i}$ using  \eqref{eq:mupd}-\eqref{eq:cupd}
    \IF{$i=0$}
    \STATE solve the optimization problem \eqref{eq:problem2}  with the control policies in \eqref{eq:liftedpolicy} to obtain  $\{\inp_t^*,\ \cdots,\ \inp_{t+N-1}^*\}$
    \ENDIF
    \STATE apply $\inp^*_{t+i}$
    \STATE calculate $\hat\state_{t+i+1|t+i}$ and $P_{t+i+1|t+i}$ using \eqref{eq:mpred}-\eqref{eq:cpred}

    \ENDFOR

\STATE set $t=t+N_c$
\ENDLOOP
\end{algorithmic}
\end{algorithm}

%--------------------------------------------------------------------------------------
\subsection*{Compact Notation}\label{sec:lifted}
%--------------------------------------------------------------------------------------
In order to state the main results in Section \ref{sec:mainresult}, it is convenient to utilize a more compact notation that describes the evolution of all signals over the entire prediction horizon $N$.

The  evolution of the system dynamics \eqref{eq:system-a}-\eqref{eq:system-b}
over a single prediction horizon $N$, starting at $t$, can be described in a \emph{lifted}
form as follows:
\begin{equation}
\label{eq:compactdyn}
\begin{aligned}
\State_t  &= \Aa\state_t+ \Ba \Inp_t + \Da \Wnoise_t \\
\Out_t  &= \Ca \State_t +\Vnoise_t
\end{aligned}
\end{equation}
where $\State_t= \smat{\begin{smallmatrix}
\state_t \\ \state_{t+1}  \\ \vdots \\ \state_{t+N}
\end{smallmatrix}}$, $\Inp_t   = \smat{\begin{smallmatrix}
\inp_t \\ \inp_{t+1} \\ \vdots \\ \inp_{t+N-1}
\end{smallmatrix}}$, $\Wnoise_t   =\smat{
\begin{smallmatrix}
\wnoise_t \\ \wnoise_{t+1} \\ \vdots \\ \wnoise_{t+N-1}
\end{smallmatrix}}$, $\Out_t   = \smat{\begin{smallmatrix}
\out_t \\ \out_{t+1}  \\ \vdots \\ \out_{t+N}
\end{smallmatrix}}$, $\Vnoise_t   = \smat{\begin{smallmatrix}
\vnoise_t \\ \vnoise_{t+1} \\ \vdots \\ \vnoise_{t+N}
\end{smallmatrix}}$, $ \Aa =
\smat{\begin{smallmatrix}
I \\ \A \\ \vdots \\ \A^N
\end{smallmatrix}}$, \[\Ba=
\smat{\begin{smallmatrix}
 \0 &\cdots &\cdots & \0 \\
\B &\ddots &&\vdots\\
\A\B & \B &\ddots & \vdots\\
\vdots && \ddots & \0\\
\A^{N-1} \B & \cdots & \A\B& \B
\end{smallmatrix}}, \Da =\smat{
\begin{smallmatrix}
\0 &  \cdots &\cdots& \0 \\
\I & \ddots & &\vdots \\
\A & \I &\ddots &\vdots \\
\vdots && \ddots & \0\\
\A^{N-1} &\cdots &\A & \I
\end{smallmatrix}},\]
and $\Ca=\mathrm{diag}\{\C,\cdots,\C\}$.
Let \[\Kal_t \Let (\A P_{t|t}\A\transp +\var{w})\C\transp(\C (\A P_{t|t}\A\transp +\var{w}) \C\transp +\var{v})^{-1}\] be the usual Kalman gain and define $\Gamma_t \Let I-\Kal_t\C$, and $\Phi_t \Let \Gamma_t \A$. Then, we can write the estimation error vector as
\begin{equation}\label{eq:augmentederror}
    \Error_t\Let \State_t-\hat\State_t = \F^{e}_{t}\error_t + \F^{\wnoise}_{t}\Wnoise_t -\F^{\vnoise}_{t}\Vnoise_t,
\end{equation}
where $\error_t=\state_t-\hat\state_t$, $ \F^{e}_{t} =
	\begin{tiny}
		\smat{I \\ \Phi_{t} \\ \Phi_{t+1}\cdot\Phi_{t}\\ \vdots \\ \Phi_{t+N-1}\cdot\Phi_{t+N-2}\cdots\Phi_{t}}
	\end{tiny}$,
\begin{align*}
    \F^{\wnoise}_{t} &=
	\begin{tiny}
	\smat{\begin{array}{l|l|l|l}\0 &   & \0 & \0 \\
    \Gamma_{t} &   & \0 & \0\\
    \Phi_{t+1}\Gamma_{t} &   & \0 & \0\\
    \vdots &  \cdots & \vdots &\vdots \\
    \Phi_{t+N-2}\cdots\Phi_{t+1}\Gamma_{t} & & \Gamma_{t+N-2} & \0 \\
    \Phi_{t+N-1}\cdots\Phi_{t+1}\Gamma_{t} &  & \Phi_{t+N-1}\Gamma_{t+N-2} & \Gamma_{t+N-1}  \end{array}}
	\end{tiny}, \\
    \F^{\vnoise}_{t} \!&=\!\!
	\begin{tiny}
	\smat{\begin{array}{l|l|l|l|l}
    \0 & \0  & &\0 & \0  \\
    \0 & \Kal_{t} &  & \0 &\0\\
    \0 & \Phi_{t+1}\Kal_{t} & &\0&\0 \\
    \vdots &\vdots &\!\!\cdots\!\!&\vdots&\vdots  \\
    \0&\Phi_{t+N-2}\cdots\Phi_{t+1}\Kal_{t} &&\Kal_{t+N-2}& \0\\
    \0 & \Phi_{t+N-1}\cdots\Phi_{t+1}\Kal_{t} &&\Phi_{t+N-1}\Kal_{t+N-2}& \Kal_{t+N-1}    \end{array}}
	\end{tiny}\!\!.
\end{align*}
Finally, the innovations sequence can be written as
\begin{equation}\label{eq:innovations}
    \Out_t-\hat\Out_t = \Ca\F^{\error}_{t}\error_t+\Ca\F^{\wnoise}_{t}\Wnoise_t+(\I-\Ca\F^{\vnoise}_{t})V_t,
\end{equation}
where $\hat\Out_t \Let \Ca\hat\State_t $. Consequently, \emph{the innovations sequence over the prediction horizon is independent of the inputs vector $\Inp_t$}.

The cost function \eqref{eq:totalcost} at time $t$ can be written as
\begin{equation}\label{eq:totalcost2}
 V_t  =  \EE_{\Y_t}\bigl[\State_t\transp \Wxa \State_t+\Inp_t\transp \Wua \Inp_t\bigr] ,
\end{equation}
where $\Wxa=\text{diag}\{\Wx_0,\cdots, \Wx_n\}$ and $\Wua=\text{diag}\{\Wu_0,\cdots, \Wu_{N-1}\}$. Also, the control policy \eqref{eq:policy} at time $t$ is given by
\begin{equation}\label{eq:liftedpolicy}
\Inp_t = \olgain_t + \clgain_t\varphi(\Out_t-\hat\Out_t), \tag{POL}
\end{equation}
where
\begin{align*}
\olgain_t \!:=\!\!\smat{
 \eta_{t} \\  \eta_{t+1} \\ \vdots \\
\eta_{t+N-1}}\!,
\varphi(\Out_t\!-\!\hat \Out_t)\!\Let\!\!\smat{\varphi_0(\out_t\!-\!\hat\out_t)\\ \vdots \\
\varphi_{N-1}(\out_{t+N-1}\!-\!\hat\out_{t+N-1})}\!,
\end{align*}
and $\clgain_t$ has the following lower block diagonal structure
\begin{equation}
\label{eq:decision-constraints}
\clgain_t \Let
\smat{  \theta_{t,t}& 0 & \hdots & 0\\
 \theta_{t+1,t} &  \theta_{t+1,t+1}&  & \vdots \\
\vdots &\vdots &\ddots & 0 \\
 \theta_{t+N-1,t} & \theta_{t+N-1,t+1} &\hdots &
 \theta_{t+N-1,t+N-1}}.
\end{equation}
Since the innovation vector $\Out_t-\hat\Out_t$ in \eqref{eq:innovations} is not a function of $\olgain_t$ and $\clgain_t$, the control inputs $\Inp_t$ in \eqref{eq:policy} remain affine in the decision variables. This fact is important to show convexity of the optimization problem, as will be seen in the next section. Finally, the constraint \eqref{eq:C1} can be rewritten as:
\begin{align}
&\Inp_t  \in {\mathbb{U}}\Let \bigl\{\xi\in\mathbb R^{Nm}\big|
\norm{\xi}_\infty \leq U_{\max}\bigr\}.\label{eq:bddu2}\tag{C1}
\end{align}

%--------------------------------------------------------------------------------------
\section{Main Results}\label{sec:mainresult}
%--------------------------------------------------------------------------------------

The optimization problem \eqref{eq:problem2} we have stated so far, even if it is successively feasible, does not in general guarantee  stability of the resulting receding horizon controller in Algorithm \ref{algo:general}. Unlike deterministic stability arguments utilized in MPC (see, for example, \citep{MayneRawlingsRaoScokaert-00}), we \emph{cannot} assume the existence of a final region that is control positively invariant and which is crucial to establish stability. This is simply due to the fact that the process  noise sequence is not assumed to have a compact domain of support. However, guided by our earlier results in \citep{ref:RamChaMilHokLyg-09}, we shall enforce an extra \emph{stability constraint} which, if feasible, can render the state of the closed-loop system mean-square bounded.

For $t\in\Nz$, the state estimate at time $t + \kappa$ given the state estimate at time $t$, the control inputs from time $t$ through $t + \kappa - 1$, and the corresponding process and measurement noise sequences can be written as
\[
	\est_{t+\kappa|t+\kappa} = A^\kappa \est_{t|t} + \Reach_\kappa(A, B)\smat{\inp_{t}\\ \vdots \\ \inp_{t+\kappa-1}} + \Xi_t,
\]
where we define
\begin{equation}
\label{eq:theuglyterm}
\begin{aligned}
	\Xi_t & \Let \smat{\A^{\kappa-1}\Kal_{t}\C\A,\A^{\kappa-2}\Kal_{t+1}\C\A, \cdots, \Kal_{t+\kappa-1}\C\A }\smat{\error_t\\ \vdots \\ \error_{t+\kappa-1}}\\ &
\quad + \smat{\A^{\kappa-1}\Kal_{t}\C,\A^{\kappa-2}\Kal_{t+1}\C, \cdots, \Kal_{t+\kappa-1}\C}\smat{\wnoise_t\\ \vdots \\ \wnoise_{t+\kappa-1}}\\
&\quad 	  + \smat{\A^{\kappa-1}\Kal_{t},\A^{\kappa-2}\Kal_{t+1}, \cdots, \Kal_{t+\kappa-1}}\smat{\vnoise_{t+1}\\ \vdots \\ \vnoise_{t+\kappa}}.
\end{aligned}
\end{equation}
\begin{scholium}
\label{sch:thehorror}
	There exists an integer $T$ and a positive constant $\thehorror=\thehorror(\var{\vnoise},\var{\wnoise},\A,\B,\C)$ such that
	\begin{equation}\label{eq:thehorror}
		\EE_{\Y_t}\left[\norm{\Xi_t} \right]\le\thehorror\qquad \text{for all } t\geq T.
	\end{equation}
\end{scholium}

	A proof of Scholium \ref{sch:thehorror} appears in Section \ref{sec:proofs}. Using the constant $\thehorror$, we now require the following ``drift condition'' to be satisfied: for  any given $\eps > 0$ and for every $t=0,N_c,2N_c,\cdots$, there exists a $\Inp_t\in\U$, such that the following condition is satisfied
	\begin{align}
	\label{eq:C4}\tag{C2}
		& \norm{\A_{2}^\kappa\esttwo_{t}+\Reach_\kappa(\A_{2},\B_{2})\smat{\inp_t\\ \vdots \\ \inp_{t+\kappa-1}}}\leq \bigl\|\esttwo_{t}\bigr\|-(\thehorror + \tfrac{\eps}{2})\;\;\text{whenever }\bigl\|\esttwo_{t}\bigr\|\geq \thehorror +\eps.
	\end{align}
	Note that $\smat{\inp_t\\ \vdots \\ \inp_{t+\kappa-1}} = (\olgain_t)_{1:\kappa m} + (\clgain_t)_{1:\kappa m}\varphi(\Out - \hat\Out)$. (For notational convenience, we have retained $\varphi(\Out - \hat\Out)$ with the knowledge that the matrix $(\clgain_t)_{1:\kappa m}$ causally selects the outputs as they become available, see \eqref{eq:decision-constraints}.) The constraint \eqref{eq:C4} pertains only to the second subsystem of the estimator, as the first subsystem ($\hat\state^{(1)}$) is Schur stable (see \eqref{eqn:Jordan} and \eqref{eqn:decomposedestimate}). We augment problem \eqref{eq:problem2} with the stability constraint \eqref{eq:C4} to obtain
	\begin{align}
		\min \Bigl\{V_t\ \Big| \text{ dynamics \eqref{eq:system-a}-\eqref{eq:system-b}, policies \eqref{eq:policy}, and constraints \eqref{eq:C1} \& \eqref{eq:C4}} \Bigr\}. \tag{OP2}\label{eq:problem1p}
	\end{align}

\begin{assumption}
\label{ass:stability1}
	In addition to Assumption \ref{ass:dynamics} and Assumption \ref{ass:noisematrices}, we stipulate that:
	\begin{enumerate}[label=(\roman*), leftmargin=*, align=right, widest=iii]
		\item The control authority $U_{\max}\geq U_{\max}^*$, where $U_{\max}^*\Let \sigma_{\min}(\Reach_\kappa(\A_2,\B_2))^{-1}\big(\thehorror + \tfrac{\eps}{2}\big)$.
    	\item The control horizon $N_c\geq \kappa$, where $\kappa$ is the reachability index.
    	\item $\bigl(\A,\var{w}^{1/2}\bigr)$ is stabilizable, and $(\A, \C)$ is observable.\AssmpEnd\label{ass:stability1:systemandnoise}

	\end{enumerate}
\end{assumption}

\begin{theorem}[Main Result]
\label{thm:main}
	Consider the system \eqref{eq:system-a}-\eqref{eq:system-b}, and suppose that Assumption \ref{ass:stability1} holds. Then:
	\begin{enumerate}[label={\rm (\roman*)}, leftmargin=*, align=right, widest=ii]
		\item For every time $t = 0, N_c, 2N_c, \cdots$, the optimization problem \eqref{eq:problem1p} is convex, feasible, and can be approximated to the following quadratic optimization problem:\label{maintheorem:optproblem}
			\begin{align}
				\minimize_{(\olgain_{t},\clgain_{t})} \quad	& 2\hat x_t\transp \Aa\transp \Wxa \Ba\olgain_{t} + 2 \tr{\Aa\transp \Wxa \Ba \clgain_{t}\LambdaPhiX_t}+ 2\tr{\clgain_{t}\transp  \Ba\transp \Wxa \Da \LambdaWPhi_t} \nn \\
													& \;\; + \olgain_{t}\transp \M_1 \olgain_{t} + 2\olgain_{t}\transp \M_1\clgain_{t}\LambdaPhi_t + \tr{\clgain_{t}\transp \M_1\clgain_{t}\LambdaPhiPhi_t}\label{eq:wICCs}\\
				\text{\rm subject \ to}\quad & \text{the structure of}\ \clgain_t\ \text{in} \ \eqref{eq:decision-constraints}, \nn\\
				&|(\olgain_t)_i|+\norm{(\clgain_t)_i}_1\varphi_{\max}\leq U_{\max}\quad \fa i=1,\cdots,Nm, \label{eq:constr:control-bounds}\\
			    & \norm{\A_{2}^\kappa\esttwo_{t}+\Reach_\kappa(\A_{2},\B_{2})(\olgain_t)_{1:\kappa m}  }+ \sqrt{n_2}\norm{\Reach_\kappa(\A_{2},\B_{2})(\clgain_t)_{1:\kappa m}}_\infty\varphi_{\max}\nn\\
				& \quad\leq \bigl\|\esttwo_{t}\bigr\|-(\thehorror + \tfrac{\eps}{2})\;\;\text{whenever }\bigl\|\esttwo_{t}\bigr\| \geq \thehorror +\eps\label{eq:constr:stab}
			\end{align}
			where $\mathcal M_1  = \Wua+ \Ba\transp \Wxa \Ba$,
			\begin{equation}
			\label{eq:Lambdas}
			\begin{aligned}
		   		\LambdaPhi_t & = \EE_{\Y_t}[\varphi(\Out-\hat\Out)], &  \LambdaPhiX_t & =\EE_{\Y_t}[\varphi(\Out-\hat\Out)\state_t\transp ],\\
			   	\LambdaWPhi_t & = \EE_{\Y_t}[\Wnoise\varphi(\Out-\hat\Out)\transp ], &  \LambdaPhiPhi_t & =\EE_{\Y_t}[\varphi(\Out-\hat\Out)\varphi(\Out-\hat\Out)\transp ].   \\
			\end{aligned}
			\end{equation}
		\item For every initial state and noise statistics $\bigl((\hat\state_0, \Sigma_{\state_0}), \Sigma_{\wnoise}, \Sigma_{\vnoise}\bigr)$, successive application of the control laws given by the optimization problem in \ref{maintheorem:optproblem} for $N_c$ steps renders the closed-loop system mean-square bounded in the following sense: there exists a constant $\gamma = \gamma(\Y_0, \kappa, \hat\state_0, \Sigma_{\state_0}, \Sigma_{\wnoise}, \Sigma_{\vnoise}, U_{\max}^*)>0$ such that\label{maintheorem:msbdd}
			\begin{equation}\label{eq:MSbound1}
		   	 \sup_{t\in\Nz}\;\EE_{\Y_0}\bigl[\norm{\state_t}^2\bigr]\leq \gamma.
			\end{equation}
	\end{enumerate}
\end{theorem}

%\begin{algorithm}[hhh]
%\caption{Basic Stochastic Receding Horizon Algorithm}\label{algo2}
%\begin{algorithmic}[1]
%\REQUIRE density $f(\state_0|\Y_{-1})=:f(x_0)$
%
%\STATE $t=0$
%
%\LOOP
%    \STATE measure $y_t$
%    \STATE update $\hat\state_t=\hat\state_{t|t}$ and $P_{t|t}$ using  \eqref{eq:mupd}-\eqref{eq:cupd}
%    \STATE solve the optimization problem \eqref{eq:problem1} with the control policies in \eqref{eq:liftedpolicy} to obtain  $\{\inp_t^*,\ \cdots,\ \inp_{t+N-1}^*\}$
%    \STATE set $i=0$
%
%    \WHILE{$i\leq N_c-1$}
%        \STATE apply $\inp^*_{t+i}$
%        \STATE calculate $\hat\state_{t+i+1|t+i}$ and $P_{t+i+1|t+i}$ using \eqref{eq:mpred}-\eqref{eq:cpred}
%        \STATE measure $y_{t+i+1}$
%        \STATE calculate $\hat\state_{t+i+1|t+i+1}$ and $P_{t+i+1|t+i+1}$ using  \eqref{eq:mupd}-\eqref{eq:cupd}
%        \STATE set $i=i+1$
%    \ENDWHILE
%
%\STATE $t=t+N_c$
%
%
%\ENDLOOP
%\end{algorithmic}
%\end{algorithm}
%
%
%
In practice, it may be also of interest to impose further some soft constraints both on the state and the input vectors. For example, one may be interested in imposing quadratic or linear constraints on the state, both of which can be captured in the following
\begin{equation}\tag{C3}\label{eq:statevar}
    \EE_{\Y_t}\left[\State_t\transp \mathcal S\State_t +\mathcal L\transp \State_t\right]\leq \alpha_t,
\end{equation}
where $\mathcal S= \mathcal S\transp\geq 0$ and $\alpha_t>0$. Moreover, expected energy expenditure constraints
can be posed as follows
\begin{equation}\tag{C4}\label{eq:inputvar}
    \EE_{\Y_t}\left[\Inp_t\transp \tilde{\mathcal S}\Inp_t\right]\leq \beta_t,
\end{equation}
where $\tilde{\mathcal S} = \tilde{\mathcal S}\transp \geq 0$ and $\beta_t>0$. In the absence of hard input constraints, such expectation-type constraints are commonly used in the stochastic MPC \citep{PrimbsSung-09,AgarwalCinquemaniChatterjeeLygeros-09} and in stochastic optimization in the form of  integrated chance constraints~\citep{ref:Han-83,ref:Han-06}. This is partly because it is not possible, without posing further restrictions on the boundedness of the process noise $w_t$, to ensure that hard constraints on the state are satisfied. For example, in the standard LQG setting nontrivial hard constraints on the system state would generally be violated with nonzero probability. Moreover, in contrast to chance constraints  where a bound is imposed on the probability of constraint violation, expectation-type constraints tend to give rise to convex optimization  problems under weak assumptions~\citep{AgarwalCinquemaniChatterjeeLygeros-09}.

We can also augment problem \eqref{eq:problem1p} with the constraints \eqref{eq:statevar} and \eqref{eq:inputvar} to obtain
	\begin{align}
		\min \Bigl\{V_t\ \Big| \text{ dynamics \eqref{eq:system-a}-\eqref{eq:system-b},  policies \eqref{eq:policy}, constraints \eqref{eq:C1}, \eqref{eq:C4}, \eqref{eq:statevar} \& \eqref{eq:inputvar}} \Bigr\}. \tag{OP3}\label{eq:problem2p}
	\end{align}
Notice that the constraints \eqref{eq:statevar} and \eqref{eq:inputvar} are not necessarily feasible at time $t$ for any choice of parameters $\alpha_t$ and $\beta_t$. As such, problem \eqref{eq:problem2p} may become infeasible over time if we simply apply Algorithm \ref{algo:general}. We therefore replace the optimization Step 7 in Algorithm \ref{algo:general} with the following subroutine for some given $\alpha_t^*$ and $\beta_t^*$ that make the constraints feasible, precision number $\delta>0$ and maximal number of iterations $\bar\nu$:

\vspace{0.3cm}
\hrule
\hrule
\vspace{0.1cm}
\noindent{\bf Subroutine 7}
\vspace{0.1cm}
\hrule

\begin{enumerate}[label={7\alph*:}, leftmargin=*, align=right]
  \item Set $\overline\alpha=\alpha_t^*$, $\underline\alpha=0$, $\overline\beta=\beta_t^*$, and $\underline\beta=0$
  \item Solve the optimization problem \eqref{eq:problem2p} using $\alpha_t=\overline\alpha$ and $\beta_t=\overline\beta$ to obtain the sequence $\{\inp_{t}^*,\inp_{t+1}^*,\cdots,\inp_{t+N-1}^*\}$
  \item Set $\nu=1$
  \item {\bf repeat}
  \item $\quad$ Set $\alpha_t=(\overline\alpha+\underline\alpha)/2$ and $\beta_t=(\overline\beta+\underline\beta)/2$
  \item $\quad$ Solve Solve the optimization problem \eqref{eq:problem2p} using the new $\alpha_t$ and $\beta_t$ to obtain \\ $\phantom{xx}$  the sequence $\{\tilde\inp_{t},\tilde\inp_{t+1},\cdots,\tilde\inp_{t+N-1}\}$
  \item $\quad$ {\bf if} step 5f is feasible {\bf then}
  \item $\quad\quad$ Set $\overline\alpha=\alpha_t$ and $\overline\beta=\beta_t$
  \item $\quad\quad$ Set $\{\inp_{t}^*,\inp_{t+1}^*,\cdots,\inp_{t+N-1}^*\}=\{\tilde\inp_{t},\tilde\inp_{t+1},\cdots,\tilde\inp_{t+N-1}\}$
  \item $\quad$ {\bf else}
  \item $\quad\quad$ Set $\underline\alpha=\alpha_t$ and $\underline\beta=\beta_t$
  \item $\quad$ {\bf end if}
  \item $\quad$ set $\nu=\nu+1$
  \item {\bf until} ($|\overline\alpha-\underline\alpha|\leq \delta$ and $|\overline\beta-\underline\beta|\leq \delta$) or $\nu>\bar\nu$

\end{enumerate}
\hrule
\vspace{0.3cm}

It may be argued that replacing Step 7 in Algorithm \ref{algo:general} with Subroutine 7 above increases the computational burden; however, the parameter $\bar \nu$ guarantees that this iterative process is halted after some prespecified number of steps in case the required precision $\delta$ is not reached in the meantime. In some instances, we may be given some fixed parameters $\hat\alpha$ and $\hat\beta$ with which the constraints \eqref{eq:statevar} and \eqref{eq:inputvar} have to be satisfied, respectively. This requirement can be easily incorporated into Subroutine 7 by setting $\underline\alpha=\hat \alpha$ and $\underline\beta=\hat \beta$ in step 7a.

\begin{assumption}\label{ass:stability2}
    At each time step $t=0,N_c,2N_c,\cdots$, the constants $\alpha^*_t$ and $\beta_t^*$ in Subroutine 7 are chosen as
        \begin{align*}
            \alpha^*_{t} & \Let 3\tr{\Aa\transp \mathcal S\Aa\EE_{\Y_t}\left[\state_t\state_t\transp\right]+\Da\transp \mathcal S\Da\var{\wnoise}}
            + 3Nm\sigma_{\max}(\Ba\transp \mathcal S\Ba)U_{\max}^2+ \mathcal L\transp \Aa\hat\state_t+\norm{\mathcal L\transp \Ba}_1U_{\max},\\
            \beta^*_{t} & \Let Nm\sigma_{\max}(\tilde{\mathcal S})U_{\max}^2.
        \end{align*}
\end{assumption}

\begin{corollary}\label{cor:main}
Consider the system \eqref{eq:system-a}-\eqref{eq:system-b}, and suppose that Assumption \ref{ass:stability1}, and \ref{ass:stability2} hold. Then:
\begin{enumerate}[label={\rm (\roman*)}, leftmargin=*, align=right, widest=ii]
	\item For every time $t=0,N_c,2N_c,\cdots$ the optimization problem \eqref{eq:problem2p} solved in Subroutine 7 is convex, feasible, and equivalent to the following quadratically constrained quadratic optimization problem:\label{maincorollary:optproblem}
\begin{align}
\minimize_{(\olgain_t,\clgain_t)} \quad  & 2\hat \state_t\transp \Aa\transp \Wxa \Ba\olgain_t+2
\tr{\Aa\transp \Wxa \Ba \clgain_t \LambdaPhiX_t}+ 2\tr{\clgain_t \transp  \Ba\transp \Wxa \Da \LambdaWPhi_t} \nn \\
&+ \olgain_t \transp \M_1 \olgain_t  +
2\olgain_t \transp \M_1\clgain_t \LambdaPhi_t+\tr{\clgain_t \transp \M_1\clgain_t \LambdaPhiPhi_t}\nn\\
\mathrm{subject \ to}\quad & \text{the structure of}\ \clgain_t \ \text{in} \ \eqref{eq:decision-constraints}, \nn\\
		&|(\olgain_t )_i|+\norm{(\clgain_t )_i}_1\varphi_{\max}\leq U_{\max}\quad \fa i=1,\cdots,Nm, \nn \\
    	& \norm{\A_{2}^\kappa\esttwo_{t}+\Reach_\kappa(\A_{2},\B_{2})(\olgain_t)_{1:\kappa m }}+ \sqrt{n_2}\norm{\Reach_\kappa(\A_{2},\B_{2})(\clgain_t)_{1:\kappa m }}_\infty\varphi_{\max}\nn\\
	    & \hspace{0.5cm}\leq \bigl\|\esttwo_{t}\bigr\|-(\thehorror + \tfrac{\eps}{2})\;\;\text{whenever }\bigl\|\esttwo_{t}\bigr\|\geq \thehorror +\eps, \ \forall i=1,\cdots, n_2\nn\\
		&   \olgain_t \transp \Ba\transp \mathcal S\Ba\olgain_t +2\olgain_t \transp \Ba\transp \mathcal S\Ba\clgain_t \LambdaPhi_t+\tr{\clgain_t \transp \Ba\transp \mathcal S\Ba\clgain_t \LambdaPhiPhi_t} \nn\\
		& \quad + 2\hat\state_t\transp  \Aa\transp \mathcal S\Ba\olgain_t  +2\tr{\Aa\transp \mathcal S\Ba\clgain_t \LambdaPhiX_t+\clgain_t \transp  \Ba\transp \mathcal S\Da \LambdaWPhi_t}\nn \\
		& \quad +\mathcal L\transp \Ba(\olgain_t +\clgain_t \LambdaPhi_t)+\tr{\Aa\transp \mathcal S\Aa \EE_{\Y_t}\left[\state_t\state_t\transp\right]}+\tr{\Da\transp \mathcal S\Da\var{w}}\nn\\
		& \quad +\mathcal L\transp \Aa\hat \state_t \leq \alpha_t,\label{eq:constraint-2}\\
		& \olgain_t \transp \tilde{\mathcal S}\olgain_t +2\olgain_t \transp \tilde{\mathcal S}\clgain_t \LambdaPhi_t+\tr{\clgain_t \transp \tilde{\mathcal S}\clgain_t \LambdaPhiPhi_t} \leq \beta_t. \label{eq:constraint-3}
\end{align}

\item For every initial state and noise statistics $\bigl((\hat\state_0, \Sigma_{\state_0}), \Sigma_{\wnoise}, \Sigma_{\vnoise}\bigr)$, successive application of the control laws given by the optimization problem in \ref{maincorollary:optproblem} for $N_c$ steps renders the closed-loop system mean-square bounded in the following sense: there exists a constant $\gamma = \gamma(\Y_0, \kappa, \hat\state_0, \Sigma_{\state_0}, \Sigma_{\wnoise}, \Sigma_{\vnoise}, U_{\max}^*)>0$ such that
	\[%begin{equation}\label{eq:MSbound2}
   	 \sup_{t\in\Nz}\;\EE_{\Y_0}\bigl[\norm{\state_t}^2\bigr]\leq \gamma.
	\]%end{equation}
\end{enumerate}
\end{corollary}

%--------------------------------------------------------------------------------------
\section{Discussion}\label{sec:discussion}
%--------------------------------------------------------------------------------------
The optimization problem \eqref{eq:problem1p} being solved in Theorem \ref{thm:main} is a quadratic program (QP), whereas the optimization problem \eqref{eq:problem2p} being solved in Corollary \eqref{cor:main} is
a quadratically constrained quadratic program (QCQP) in the optimization parameters $ (\eta, \Theta)$. As such, both can be easily solved via software packages such as \texttt{cvx}~\citep{ref:boydCVX} or \texttt{yalmip} \citep{YALMIP}.

%--------------------------------------------------------------------------------------
\subsection{Other Constraints and Policies}\label{sec:otherconstraints}
%--------------------------------------------------------------------------------------
It is not difficult to see that constraints on the variation of the inputs of the form
\[\norm{Pu}_\infty\leq \Delta U_{\max},\]
where $ P=\begin{tiny}\smat{I & -I& 0 &\cdots & 0 \\ 0 & I& -I&\cdots &0\\ \vdots &&&& \vdots \\ 0&\cdots &0& I&-I}\end{tiny}$, can be incorporated into the optimization problems \eqref{eq:problem1p} and \eqref{eq:problem2p}.
Moreover, we can easily solve the problem using quadratic policies of the form
\begin{equation}\label{eq:quadraticpolicy}\tag{POL$'$}
    \Inp_t=\eta+\Theta\varphi(\Out_t-\hat\Out_t)+\tilde\Theta\tilde\varphi(\Out_t-\hat\Out_t),
\end{equation}
instead of \eqref{eq:policy}, where
\begin{equation*}
    {\tiny \tilde\Theta\Let \smat{\tilde\theta_{0,0}&0&\cdots & 0\\ \tilde\theta_{1,0} & \tilde\theta_{1,1} &&\vdots \\ \vdots & \vdots & \ddots & 0\\
    \tilde\theta_{N-1,0} & \tilde\theta_{N-1,1} & \cdots & \tilde\theta_{N-1,N-1}}}, \ \tilde\varphi(z)\Let {\tiny \smat{\tilde\varphi_0(\out_t-\hat\out_t)\transp \tilde\varphi_0(\out_t-\hat\out_t)\\ \vdots \\ \tilde\varphi_{N-1}(\out_{t+N-1}-\hat\out_{t+N-1})\transp \tilde\varphi_{N-1}(\out_{t+N-1}-\hat\out_{t+N-1})}},
\end{equation*}
and $\tilde\theta_{i,j}\in\R^{m\times 1}$. The underlying optimization problems  \eqref{eq:problem1p} and \eqref{eq:problem2p} with the policy \eqref{eq:quadraticpolicy} are still convex and both Theorem \eqref{thm:main} and Corollary \eqref{cor:main} still apply with minor changes.

%--------------------------------------------------------------------------------------
\subsection{Off-Line Computation of the $\Lambda$ Matrices}\label{sec:computational}
%--------------------------------------------------------------------------------------
At any time $t=0,N_c,2N_c,\cdots$, our ability to solve the optimization problems \eqref{eq:problem1p} and \eqref{eq:problem2p} in Theorem \ref{thm:main} and Corollary \ref{cor:main}, respectively, hinges upon our ability to compute the following matrices
\begin{equation}
\label{eq:genericLambda}
\begin{aligned}
\LambdaPhiE_t& =\EE_{\Y_t}[\varphi(\Out_t-\hat\Out_t)(\state_t-\hat{\state}_{t})\transp ], & \LambdaWPhi_t & = \EE_{\Y_t}[\Wnoise\varphi(\Out_t-\hat\Out_t)\transp ],\\
\LambdaPhi_t& =\EE_{\Y_t}[\varphi(\Out_t-\hat\Out_t)], & \LambdaPhiPhi_t & =\EE_{\Y_t}[\varphi(\Out_t-\hat\Out_t)\varphi(\Out_t-\hat\Out_t)\transp ],\\
\LambdaPhiX_t & =\LambdaPhiE_t + \LambdaPhi_t\hat{x}_{t}\transp.  & &
\end{aligned}
\end{equation}
Recall that $\Out_t-\hat\Out_t$  is the innovations sequence that was given in \eqref{eq:innovations}, and $\hat\state_t$ is the optimal mean-square estimate of $\state_t$ given the history $\Y_t$. The matrices \eqref{eq:genericLambda} may be
computed by numerical integration with respect to the independent
Gaussian measures of $w_{t},\ldots w_{t+N-1}$, of $v_{t},\ldots$ $
v_{t+N}$, and of $x_t$ given $\Y^t$. Due to the large
dimensionality of the integration space, this approach may be
impractical for online computations. One alternative approach relies
on the observation that
$\LambdaPhiE_t,\LambdaPhi_t,\LambdaWPhi_t$, and $\LambdaPhiPhi_t$ depend on $\state_t$ via the difference
$\state_t-\hat\state_{t}$. Since $\state_t-\hat\state_{t}$ is conditionally
zero-mean given $\Y^t$, we can write the dependency
of~\eqref{eq:genericLambda} on the time-varying statistics of $x_t$ given $\Y^t$ as follows:
\begin{equation}
\label{eq:paramlambda}
\begin{aligned}
&\LambdaPhiX_t(\hat{x}_{t},P_{t|t})=\LambdaPhiE_t(P_{t|t})+ \LambdaPhi_t(P_{t|t})\hat{x}_{t}\transp ,\\
& \LambdaWPhi_t(P_{t|t}), \text{ and } \LambdaPhiPhi_t(P_{t|t}).
\end{aligned}
\end{equation}
In principle one may compute \emph{off-line} and store the matrices $\LambdaPhiE_t(P_{t|t}),\LambdaPhi_t(P_{t|t}),\LambdaWPhi_t(P_{t|t})$, and $\LambdaPhiPhi_t(P_{t|t})$, which depend on the covariance matrices $P_{t|t}$, and just update online the value of $\LambdaPhiX_t(\hat\state_t,P_{t|t})$ as the estimate $\hat\state_t$ becomes available. However, this poses serious requirements in terms of memory.
A more appealing alternative is to exploit the convergence
properties of $P_{t|t}$. The following result can be inferred, for instance, from \cite[Theorem 5.1]{KamenSu-99}.
\begin{proposition}\label{prop:ARE}
Under Assumptions \ref{ass:noisematrices} and \ref{ass:stability1}-\ref{ass:stability1:systemandnoise}
\begin{itemize}
  \item the (discrete-time)
algebraic Riccati equation in $P\in\RR^{n\times n}$
  \begin{equation}\label{eq:Riccatipred}
     P=\A [P-P \C\transp (\C P\C\transp +\var{v})^{-1}\C P]\A\transp +\var{w}
  \end{equation}
   has a unique solution $P^*\geq 0$, and
  \item the sequence $P_{t+1|t}$ defined by~\eqref{eq:cupd}
and~\eqref{eq:cpred} converges to $P^*$ as $t$ tends to $\infty$, for any initial condition $P_{0|-1}\geq 0$.
\end{itemize}
\end{proposition}
The assumption that $\var{v}>0$ can be relaxed to $\var{v}\geq 0$ at the price
of some additional technicality (more on this can be found in \citep{FerrantePicciPinzoni-02}). As a consequence of this result, under detectability and stabilizability assumptions, $P_{t|t}$
converges to
\begin{equation}
\label{eq:Riccatifilter}
P^\circ=P^*-P^*\C\transp  (\C P^* \C\transp  +\var{v})^{-1}\C P^*,
\end{equation}
which is the asymptotic error covariance matrix of the estimator
$\hat{x}_{t}$. Thus, neglecting the initial transient, it makes
sense to just compute off-line and store the matrices $\LambdaPhiE_t(P^\circ),\LambdaPhi_t(P^\circ),\LambdaWPhi_t(P^\circ)$, and $\LambdaPhiPhi_t(P^\circ)$, and just update the matrix
$\LambdaPhiX_t$ for new values of the estimate $\hat\state_t$.

%--------------------------------------------------------------------------------------
\section{Proofs}\label{sec:proofs}
%--------------------------------------------------------------------------------------
The proofs of the main results are presented as follows: We begin by showing the result in Scholium \ref{sch:thehorror}. Then, we state a fundamental result pertaining to mean-square boundedness for general random sequences in Proposition \ref{p:PR99}, which is utilized to show the mean-square boundedness conclusions of Theorem \ref{thm:main} and Corollary \ref{cor:main}. We proceed to show the first assertion \ref{maintheorem:optproblem} of Theorem \ref{thm:main} in Lemma \ref{l:maintheorem:opt}. The proof of the second assertion \ref{maintheorem:msbdd} of Theorem \ref{thm:main} starts by showing Lemmas \ref{l:PRcondition1} and \ref{l:PRcondition2} to conclude mean-square boundedness of the orthogonal subsystem $(A_2,B_2)$. We conclude the proof of Theorem \ref{thm:main} by showing mean-square boundedness of the Schur stable subsystem $(A_1,B_1)$. We end this section by proving the extra conclusions of Corollary \ref{cor:main}, beyond those present in Theorem \ref{thm:main}.

Let us now look at the estimation equation $\hat\state_t (=\hat\state_{t|t})$ in \eqref{eq:mupd} and combine it with \eqref{eq:mpred} and the output equation \eqref{eq:system-b} to obtain
	\begin{equation}
	\label{e:fundamental}
		\est_{t+1} = \A \est_{t} + \B \inp_t + \Kal_t\bigl(\C\A(\state_t - \est_{t}) + \C\wnoise_t + \vnoise_{t+1}\bigr),
	\end{equation}
	where $\Kal_t = (\A \Pmat_{t|t} \A\transp + \Sigma_{\wnoise})\C\transp\bigl(\C (\A \Pmat_{t|t} \A\transp + \Sigma_{\wnoise}) \C\transp + \Sigma_{\vnoise}\bigr)^{-1}$ is the Kalman gain. Our next Fact pertains to the boundedness of the error covariance matrices $P_{t|t}$ in \eqref{eq:cupd}.
	\begin{fact}\label{fact:factsonbounds}
		There exists $\validfrom\in\Nz$ and $\partialhorror, \partialhorror_m > 0$ such that $\tr{P_{t|t}} \le \partialhorror$ for all $t\ge\validfrom$, and $\norm{\Kal_t} \le \partialhorror_m$ for all $t\ge\validfrom$.
	\end{fact}
\noindent Fact \ref{fact:factsonbounds} follows, for example, immediately from Lemma \ref{prop:Pbounds} in Appendix \ref{appendixB} and since by assumption $\Sigma_{\vnoise} > 0$, one possible bound on $\bigl(\norm{\Kal_t}\bigr)_{t\ge\validfrom}$ is given by $\norm{\Kal_t} \le \frac{\norm{\A \Pmat_{t|t} \A\transp + \Sigma_{\wnoise}}\norm{C\transp}}{\lambda_{\min}(\Sigma_{\vnoise})}$ $\le \frac{\bigl(\norm{\Sigma_{\wnoise}} + \norm{A}^2\norm{P_{t|t}}\bigr)\norm{C\transp}}{\lambda_{\min}(\Sigma_{\vnoise})}$. Note that there are many alternative bounds in the Literature, see, for example, \citep{ref:AndersonMooreFiltering,ref:Jaz-70}. Now, using the bounds in Fact \ref{fact:factsonbounds}, we can proceed to prove Scholium \ref{sch:thehorror}.

\begin{proof}[Proof of Scholium \ref{sch:thehorror}]
Recall the expression of $\Xi_{t}$ in \eqref{eq:theuglyterm} and define the following quantities:
\begin{equation}
\label{e:variousH}
\begin{aligned}
F^{\error}_{\kappa t} & \Let
\begin{bmatrix}
	\A^{\kappa-1} \Kal_{\kappa t} \C\A & \cdots & \A \Kal_{\kappa(t+1)-2}\C\A & \Kal_{\kappa(t+1)-1}\C\A
\end{bmatrix},\\
F^{\wnoise}_{\kappa t} & \Let
\begin{bmatrix}
	\A^{\kappa-1} \Kal_{\kappa t} \C & \cdots & \A \Kal_{\kappa(t+1)-2}\C & \Kal_{\kappa(t+1)-1}\C
\end{bmatrix},\\
F^{\vnoise}_{\kappa t} & \Let
\begin{bmatrix}
	\A^{\kappa-1} \Kal_{\kappa t}  & \cdots & \A \Kal_{\kappa(t+1)-2} & \Kal_{\kappa(t+1)-1}
\end{bmatrix}.
\end{aligned}
\end{equation}	
Then, $\Xi_{\kappa t}$ can be rewritten as
\begin{equation}\label{eq:thenotsouglyterm}
    \Xi_{\kappa t} = F^{\error}_{\kappa t} \error_{\kappa t:\kappa(t+1)-1} + F^{\wnoise}_{\kappa t}\wnoise_{\kappa t:\kappa(t+1)-1}+ F^{\vnoise}_{\kappa t}\vnoise_{\kappa t+1:\kappa(t+1)}.
\end{equation}
But for $t\ge\kjumpvalidfrom$, \footnote{Recall that for any positive real number $s$, $\lceil s\rceil$ denotes the smallest integer that upper-bounds $s$.} we have that
\begin{align*}
\norm{F^{\error}_{\kappa t}} & \le \kappa \norm{\C\A} \sup_{\ell\ge\kappa\kjumpvalidfrom}\norm{\Kal_\ell},\\
\norm{F^{\wnoise}_{\kappa t}} & \le \kappa \norm{\C} \sup_{\ell\ge\kappa\kjumpvalidfrom}\norm{\Kal_\ell},\\
\norm{F^{\vnoise}_{\kappa t}} & \le \kappa \sup_{\ell\ge\kappa\kjumpvalidfrom}\norm{\Kal_\ell}.
\end{align*}
Using Fact \ref{fact:factsonbounds} we know that $\sup_{\ell\ge\kappa\kjumpvalidfrom}\norm{\Kal_\ell} \le \partialhorror_h$. Thus, it suffices to take
\[
\thehorror = \kappa^{3/2}\partialhorror_h \Bigl(\norm{\C\A}\sqrt{\partialhorror} + \norm{C}\sqrt{\tr{\Sigma_\wnoise}} + \sqrt{\tr{\Sigma_\vnoise}}\Bigr)
\]
in order to upper-bound the expectation of $\Xi(t)$ in \eqref{eq:thenotsouglyterm} after time $T\Let\kappa\kjumpvalidfrom$.
\end{proof}

The following result pertains to mean-square boundedness of a random sequence $(\xi_t)_{t\in\Nz}$:
	\begin{proposition}[{\cite[Theorem~1]{ref:PemRos-99}}]
	\label{p:PR99}
	    Let $(\xi_t)_{t\in\Nz}$ be a sequence of nonnegative random variables on some probability space $(\Omega, \sigalg, \PP)$, and let $(\sigalg_t)_{t\in\Nz}$ be any filtration to which $(\xi_t)_{t\in\Nz}$ is adapted. Suppose that there exist constants $b > 0$, and $J, M < \infty$, such that
		\[
			\xi_0\le J,
		\]
		and for all $t\in\Nz$:
	    \begin{gather}
	        \EE\bigl[\xi_{t+1} - \xi_t\big|\sigalg_t\bigr] \le -b\quad \text{on the event }\{\xi_t > J\},\quad\text{and}\label{e:FLcond}\\
	        \EE\bigl[\abs{\xi_{t+1} - \xi_t}^4\big|\xi_0,\ldots, \xi_t\bigr] \le M.\label{e:pcond}
	    \end{gather}
	    Then there exists a constant $\gamma = \gamma(b, J, M) > 0$ such that $\displaystyle{\sup_{t\in\Nz}\EE\bigl[\xi_t^{2}\bigr] \le \gamma}$.
	\end{proposition}
\noindent

	\begin{lemma}
	\label{l:maintheorem:opt}
		Consider the system \eqref{eq:system-a}-\eqref{eq:system-b}, and suppose that Assumption \ref{ass:stability1} holds. Then the first assertion \ref{maintheorem:optproblem} of Theorem \ref{thm:main} holds.
	\end{lemma}
	\begin{proof}
		It is clear that $\State_t\transp \Wxa \State_t + \Inp_t\transp \Wua \Inp_t$ is convex quadratic in $\State_t$ and $\Inp_t$, and both $\State_t$ and $\Inp_t$ are affine functions of the design parameters $(\olgain_t, \clgain_t)$ for every realization of the noise $(\wnoise_t)_{t\in\Nz}$. Since taking expectation of a convex function retains convexity \citep{ref:BoyVan-04}, we conclude that $\EE_{\state_t}\bigl[\State_t\transp \Wxa \State_t + \Inp_t\transp \Wua \Inp_t\bigr]$ is convex in $(\olgain_t, \clgain_t)$. Also, note that the constraint sets described by \eqref{eq:constr:control-bounds} and \eqref{eq:constr:stab} are convex in $(\olgain_t, \clgain_t)$. This settles the claim concerning convexity of the optimization program in Theorem \ref{thm:main}-\ref{maintheorem:optproblem}.

		Concerning the objective function \eqref{eq:wICCs}, we have that
		\begin{align*}
			& \EE_{\Y_t}\bigl[\State_t\transp \Wxa \State_t + \Inp_t\transp \Wua \Inp_t\bigr]\\
			&\; = \EE_{\Y_t}\bigl[\bigl(\Aa\state_t + \Ba\Inp_t + \Da\Wnoise_t\bigr)\transp \Wxa \bigl(\Aa\state_t + \Ba\Inp_t + \Da\Wnoise_t\bigr) + \Inp_t\transp \Wua \Inp_t\bigr]\\
			&\; = \EE_{\Y_t}\Bigl[\bigl\|\Aa\state_t + \Ba\bigl(\olgain_t + \clgain_t\varphi(\Out_t - \hat\Out_t)\bigr) + \Da\Wnoise_t\bigr\|_{\Wxa}^2 + \bigl\|\olgain_t + \clgain_t\varphi(\Out_t - \hat\Out_t)\bigr\|_{\Wua}^2\Bigr]\\
			&\; = \EE_{\Y_t}\Bigl[\state_t\transp\Aa\transp\Wxa\Aa\state_t + 2\state_t\transp \Aa\transp\Wxa \Ba \olgain_t + 2 \state_t\transp\Aa\transp\Wxa \Ba \clgain_t\varphi(\Out_t - \hat\Out_t) + 2 \state_t\transp\Aa\transp\Wxa\Da \Wnoise_t\\
			& \qquad + \olgain_t\transp(\Ba\transp\Wxa\Ba + \Wua)\olgain_t + 2\olgain_t\transp(\Ba\transp\Wxa\Ba + \Wua)\clgain_t\varphi(\Out_t - \hat\Out_t) + 2\olgain_t\transp\Ba\transp\Wxa\Da\Wnoise_t\\
			& \qquad + \varphi(\Out_t - \hat\Out_t)\transp\clgain_t\transp(\Ba\transp \Wxa \Ba + \Wua)\clgain_t\varphi(\Out_t - \hat\Out_t) + 2\varphi(\Out_t - \hat\Out_t)\transp\clgain_t\transp\Ba\transp\Wxa\Da\Wnoise_t\\
			& \qquad + \Wnoise_t\transp\Da\transp\Wxa\Da\Wnoise_t\Bigr]
		\end{align*}
		Since $\EE_{\{\Y_t, \state_t\}}\bigl[\Wnoise_t\bigr] = 0$ and
		in view of the definitions of the various matrices in \eqref{eq:Lambdas} we have
		\begin{align*}
			V_t & = 2 \hat x_t\transp \Aa\transp \Wxa \Ba\olgain_t + 2 \tr{\Aa\transp \Wxa \Ba \clgain_t\LambdaPhiX_t}+ 2\tr{\clgain_t\transp  \Ba\transp \Wxa \Da \LambdaWPhi_t}\\
			& \quad + \olgain_t\transp \M_1 \olgain_t + 2\olgain_t\transp \M_1\clgain_t\LambdaPhi_t + \tr{\clgain_t\transp \M_1\clgain_t\LambdaPhiPhi_t}\\
			& \quad + \text{terms that do not depend on }(\olgain_t, \clgain_t).
		\end{align*}
This tallies the objective function in \eqref{eq:wICCs} with the objective $V_t$ in \eqref{eq:totalcost2}.

Concerning the constraints, we have shown in \citep{ref:HokChaLyg-09,ref:ChaHokLyg-09} that combining the constraint $\norm{\inp_t}_\infty \le U_{\max}$ and the class of policies \eqref{eq:liftedpolicy} is equivalent to the constraints $\abs{(\olgain_t)_i} + \norm{(\clgain_t)_i}_1\varphi_{\max} \le U_{\max}$ for all $i = 1, \ldots, Nm$, which accounts for the constraint \eqref{eq:constr:control-bounds}. Substituting \eqref{eq:liftedpolicy} into the stability constraint \eqref{eq:C4}, we obtain
\begin{align*}
& \norm{\A_{2}^\kappa\esttwo_{t}+\Reach_\kappa(\A_{2},\B_{2})(\olgain_t)_{1:\kappa m}+\Reach_\kappa(\A_{2},\B_{2})(\clgain_t)_{1:\kappa m}\varphi(\Wnoise)}\\
&\qquad \leq \norm{\A_{2}^\kappa\esttwo_{t}+\Reach_\kappa(\A_{2},\B_{2})(\olgain_t)_{1:\kappa m}}+ \norm{\Reach_\kappa(\A_{2},\B_{2})(\clgain_t)_{1:\kappa m}\varphi(\Wnoise)}\\
&\qquad \leq \norm{\A_{2}^\kappa\esttwo_{t}+\Reach_\kappa(\A_{2},\B_{2})(\olgain_t)_{1:\kappa m}}+ \sqrt{n_2}\norm{\Reach_\kappa(\A_{2},\B_{2})(\clgain_t)_{1:\kappa m}\varphi(\Wnoise)}_\infty\\
&\qquad \leq \norm{\A_{2}^\kappa\esttwo_{t}+\Reach_\kappa(\A_{2},\B_{2})(\olgain_t)_{1:\kappa m}}+ \sqrt{n_2}\norm{\Reach_\kappa(\A_{2},\B_{2})(\clgain_t)_{1:\kappa m}}_\infty\varphi_{\max}.
\end{align*}
Accordingly, if condition constraint  \eqref{eq:constr:stab} is satisfied, then the stability constraint \eqref{eq:C4} is satisfied as well.

It remains to show that the constraints are simultaneously feasible. Inspired by the work in \citep{ref:RamChaMilHokLyg-09}, we consider the candidate controller
\begin{equation*}
   \tilde u_{t:t+\kappa-1}\Let \smat{\eta_t\\ \vdots \\ \eta_{t+\kappa-1}}\Let -\Reach_\kappa(\A_2,\B_2)^\dagger \sat_r(\A_2^\kappa\hat\state_t^{(2)}),
\end{equation*}
i.e., with $\clgain_t\equiv 0$, where $r\Let \thehorror+\eps/2$. First, we have that
\[\norm{\tilde u_{t:t+\kappa-1}}_\infty\leq \norm{\tilde u_{t:t+\kappa-1}}_2\leq \sigma_{\min}(\Reach_\kappa(\A_2,\B_2))^{-1}(\thehorror+\eps/2)= U_{\max}^*,\]
and the constraint \eqref{eq:C1} is satisfied. Concerning constraint \eqref{eq:C4}, we have that
\begin{align}
 \norm{\A_2^\kappa\hat\state_t^{(2)}+\Reach_\kappa(\A_2,\B_2)\tilde u_{t:t+\kappa-1}}  &= \norm{\hat\state_t^{(2)}} -r\nn\\
 &=\norm{\hat\state_t^{(2)}} -(\thehorror+\eps/2), \ \ \text{whenever }\bigl\|\esttwo_{t}\bigr\|\geq \thehorror +\eps, \label{eq:feasibilitybound}
\end{align}
where the first equality follows from the orthogonality of $\A_2$ (see \citep{ref:RamChaMilHokLyg-09}),
and the constraint \eqref{eq:constr:stab} is also satisfied.
The optimization program \eqref{eq:wICCs} subject to \eqref{eq:constr:control-bounds}-\eqref{eq:constr:stab} is therefore a quadratic program that is equivalent to \eqref{eq:problem1p}.
	\end{proof}

	\begin{lemma}
	\label{l:PRcondition1}
		Consider the system \eqref{eq:system-a}-\eqref{eq:system-b}, and suppose that Assumption \ref{ass:stability1} holds. Then there exist constants $b, J > 0$ such that\footnote{For $z\in\R$ the notation $\ceil z$ stands for the smallest integer greater or equal to $z$.}
		\[
			\EE_{\Y_{\kappa t}}\Bigl[\bigl\|\esttwo_{\kappa (t+1)}\bigr\| - \bigl\|\esttwo_{\kappa t}\bigr\|\Bigr] \le -b \quad\text{on the set }\bigl\{\bigl\|\esttwo_{\kappa t}\bigr\| > J\bigr\}\qquad\text{for all }t\ge\kjumpvalidfrom.
		\]
	\end{lemma}
	\begin{proof}
	Let $\Pi^{(2)}$  be a projection operator that picks the last $n_2$ components of any vector in $\R^n$, and	 consider the subsampled process $(\esttwo_{t\kappa})_{t\in\Nz}$ given by
\begin{equation}
\label{e:subsampledequation}
	\esttwo_{\kappa (t+1)}  = \A_2^\kappa \esttwo_{\kappa t} + \Reach_\kappa(\A_2, \B_2)\inp_{\kappa t:\kappa(t+1)-1}
		 + \Pi^{(2)}\bigl(\Xi_{\kappa t}\bigr).
\end{equation}
By utilizing the triangle inequality we get
		\begin{align*}
			& \EE_{\Y_{\kappa t}}\Bigl[\norm{\esttwo_{\kappa(t+1)}} - \norm{\esttwo_{\kappa t}}\Bigr]\\
			& \; \le \EE_{\Y_{\kappa t}}\Bigl[\Bigl\|\A_2^\kappa \esttwo_{\kappa t} + \Reach_\kappa(\A_2, \B_2)\inp_{\kappa t:\kappa(t+1)-1}\Bigr\| -  \norm{\esttwo_{\kappa t}}\Bigr] +
\EE_{\Y_{\kappa t}}\Bigl[\norm{\Pi^{(2)}\bigl(\Xi_{\kappa t}\bigr)}\Bigr].
		\end{align*}
		We know from Scholium \ref{sch:thehorror} that there exists  a uniform (with respect to time $t$) upper bound $\thehorror$ for the last term on the right-hand side of the preceding inequality for $t\ge\kappa\kjumpvalidfrom$. We rewrite the inequality as
		\begin{align*}
			\EE_{\Y_{\kappa t}}\Bigl[\norm{\esttwo_{\kappa(t+1)}} - \norm{\esttwo_{\kappa t}}\Bigr] &\le \EE_{\Y_{\kappa t}}\Bigl[\norm{\A_2^\kappa \esttwo_{\kappa t} + \Reach_\kappa(\A_2, \B_2)\inp_{\kappa t:\kappa(t+1)-1}} - \norm{\esttwo_{\kappa t}}\Bigr] + \thehorror\\
 & \leq - \frac{\eps}{2},\qquad \text{whenever }\bigl\|\esttwo_{t}\bigr\|\geq \thehorror +\eps,
		\end{align*}
where the last inequality follows from \eqref{eq:feasibilitybound}. Setting $b=\eps/2$ and $J=\thehorror+\eps$ completes the proof.
	\end{proof}

	\begin{lemma}
	\label{l:PRcondition2}
		Consider the system \eqref{eq:system-a}-\eqref{eq:system-b}, and suppose that Assumption \ref{ass:stability1} holds. Then there exists a constant $M > 0$ such that
		\[
			\EE\Bigl[\abs{\norm{\esttwo_{\kappa(t+1)}} - \norm{\esttwo_{\kappa t}}}^4\,\Big|\,\norm{\esttwo_{\kappa i}},\;i = \kjumpvalidfrom, \ldots, t\Bigr] \le M\qquad \text{for all }t\ge\kjumpvalidfrom.
		\]
	\end{lemma}
	\begin{proof}
		Fix $t\ge\kjumpvalidfrom$, and consider again the subsampled process in \eqref{e:subsampledequation}. We have, by orthogonality of $\A_2$, that $\norm{\esttwo_{\kappa t}} = \norm{\A_2^\kappa\esttwo_{\kappa t}}$ and, by the triangle inequality, that
		\begin{align*}
			& \EE\Bigl[\abs{\norm{\esttwo_{\kappa(t+1)}} - \norm{\esttwo_{\kappa t}}}^4\,\Big|\,\norm{\esttwo_{\kappa i}},\;i = \kjumpvalidfrom, \ldots, t\Bigr]\nn\\
			& = \EE\Bigl[\abs{\norm{\esttwo_{\kappa(t+1)}} - \norm{\A_2^\kappa \esttwo_{\kappa t}}}^4\,\Big|\,\norm{\esttwo_{\kappa i}},\;i = \kjumpvalidfrom, \ldots, t\Bigr]\nn\\
			& \le \EE\Bigl[\norm{\A_2^\kappa \esttwo_{\kappa t} + \Reach_\kappa(\A_2, \B_2)\inp_{\kappa t:\kappa(t+1)-1} + \Xi_{\kappa t} - \A_2^\kappa \esttwo_{\kappa t}}^4\,\Big|\,\norm{\esttwo_{\kappa i}},\;i = \kjumpvalidfrom, \ldots, t\Bigr]\nn\\
			& = \EE\Bigl[\Bigl(\norm{\Reach_\kappa(\A_2, \B_2)\inp_{\kappa t:\kappa(t+1)-1}} + \norm{\Xi_{\kappa t}}\Bigr)^4\,\Big|\,\norm{\esttwo_{\kappa i}},\;i = \kjumpvalidfrom, \ldots, t\Bigr].\label{e:fourthmomentbnd}
		\end{align*}
Recall that for any two positive numbers $a$ and $b$, $(a+b)^2\leq 2a^2+2b^2$. Using this latter fact, we arrive at the following upper bound
		\begin{align*}
			& \EE\Bigl[\abs{\norm{\esttwo_{\kappa(t+1)}} - \norm{\esttwo_{\kappa t}}}^4\,\Big|\,\norm{\esttwo_{\kappa i}},\;i = \kjumpvalidfrom, \ldots, t\Bigr]\nn\\
			& \leq 8 \EE\Bigl[\norm{\Reach_\kappa(\A_2, \B_2)\inp_{\kappa t:\kappa(t+1)-1}}^4+\norm{\Xi_{\kappa t}}^4\,\Big|\,\norm{\esttwo_{\kappa i}},\;i = \kjumpvalidfrom, \ldots, t\Bigr].
		\end{align*}
		By design $\norm{u_i}_\infty \le U_{\max}$ and $\Xi_{\kappa t}$ is Gaussian and independent of $\Bigl\{\norm{\esttwo_{\kappa i}},\;i = \kjumpvalidfrom, \ldots, t\Bigr\}$ and has its fourth moment bounded. Therefore, we can easily infer that there exists an $M>0$ such that
		\[
			\EE\Bigl[\Bigl(\norm{\Reach_\kappa(\A_2, \B_2)\inp_{\kappa t:\kappa(t+1)-1}} + \norm{\Xi_{\kappa t}}\Bigr)^4\,\Big|\,\norm{\esttwo_{\kappa i}},\;i = \kjumpvalidfrom, \ldots, t\Bigr] \le M
		\]
		for all $t\ge\kjumpvalidfrom$.
	\end{proof}

	We are finally ready to prove Theorem \ref{thm:main}.

	\begin{proof}[Proof of Theorem \ref{thm:main}]
		Claim  \ref{maintheorem:optproblem} of Theorem \ref{thm:main} was proved in Lemma \ref{l:maintheorem:opt}. It remains to show claim \ref{maintheorem:msbdd}. We start by asserting the following inequality
\begin{equation*}
			\EE_{\Y_{\kappa\kjumpvalidfrom}}\bigl[\norm{\state_t}^2\bigr] \le 2 \EE_{\Y_{\kappa\kjumpvalidfrom}}\bigl[\norm{\state_t - \hat\state_{t}}^2\bigr] + 2\EE_{\Y_{\kappa\kjumpvalidfrom}}\bigl[\norm{\hat\state_{t}}^2\bigr].
\end{equation*}
We know from Fact \ref{fact:factsonbounds} that $\EE_{\Y_{\kappa\kjumpvalidfrom}}\bigl[\norm{\state_t - \hat\state_t}^2\bigr] \le \partialhorror$ for all $t \ge \kappa\kjumpvalidfrom$. As such, if we are able to show that the state of the estimator is mean-square bounded, we can immediately infer a mean-square bound on the state of the plant.

We first start by splitting the squared norm of the estimator state as $\norm{\hat\state_t}^2 = \bigl\|\hat\state_t^{(1)}\bigr\|^2 + \bigl\|\hat\state_t^{(2)}\bigr\|^2$, where $\hat\state^{(1)}$ and $\hat\state^{(2)}$ are states corresponding to the Schur and orthogonal parts of the system, respectively. It then follows that
		\begin{equation}
		\label{e:estdecomposition}
			\EE_{\Y_{\kappa\kjumpvalidfrom}}\bigl[\norm{\hat\state_t}^2\bigr] = \EE_{\Y_{\kappa\kjumpvalidfrom}}\bigl[\bigl\|\hat\state_t^{(1)}\bigr\|^2\bigr] +\EE_{\Y_{\kappa\kjumpvalidfrom}}\bigl[ \bigl\|\hat\state_t^{(2)}\bigr\|^2\bigr],\qquad t\in\Nz.
		\end{equation}
Letting $\xi_t \Let \bigl\|\esttwo_t\bigr\|$ for $t\in\Nz$, we see from Lemma \ref{l:PRcondition1} and Lemma \ref{l:PRcondition2} the conditions  \eqref{e:FLcond} and \eqref{e:pcond} of Proposition \ref{p:PR99} are verified for the sequence $(\xi_{\kappa t})_{t\ge\kjumpvalidfrom}$. Thus, by Proposition \ref{p:PR99}, there exists a constant $\gamma^{(2)} = \gamma^{(2)}(b, J, M) > 0$ such that
		\[
			\EE_{\Y_{\kappa\kjumpvalidfrom}}\bigl[\bigl\|\esttwo_{\kappa t}\bigr\|^2\bigr] \le \gamma^{(2)}\qquad \text{for all }t\ge\kjumpvalidfrom.
		\]
Hence, the orthogonal subsystem is mean-square bounded.

Since the matrix $\A_1$ is Schur stable, we know \citep[Proposition 11.10.5]{ref:Ber-09} that there exists a positive definite matrix $\Pmatone\in\R^{n_1\times n_1}$ that satisfies $\A_1\transp \Pmatone \A_1 - \Pmatone = - \Qmatone$. It easily follows that there exists $\rho\in\;]0, 1[$ such that $\A_1\transp \Pmatone \A_1 - \Pmatone \le -\rho \Pmatone$; in fact, $\rho$ can be chosen from the interval $\bigl]0, \min\bigl\{1, \lambda_{\min}\bigl(\Qmatone\bigr)/\lambda_{\max}\bigl(\Pmatone\bigr)\bigr\}\bigr[$. Therefore,  we see that for any $t\ge\validfrom$,
		\begin{align*}
			\EE_{\Y_t}& \left[\norm{\estone_{t+1}}^2_{\Pmatone}\right] - \norm{\estone_{t}}^2_{\Pmatone} \le  - \rho\norm{\estone_{t}}^2_{\Pmatone} + 2\EE_{\Y_t}\bigl[(\estone_t)\transp\A_1\transp\Pmatone\B_1\inpone_t\bigr] + \EE_{\Y_t}\bigl[\Pi^{(1)}(\Xi_t)\bigr],
		\end{align*}
		where $\Pi^{(1)}$ is the projection onto the first $n_1$ coordinates and $\Xi_t$ was defined in \eqref{eq:theuglyterm}. Young's inequality shows that $2\EE_{\Y_t}\bigl[(\estone_t)\transp\A_1\transp\Pmatone\B_1\inpone_t\bigr] \le \epsilon\EE_{\Y_t}\bigl[\bigl\|\A_1\estone_t\bigr\|_{\Pmatone}^2\bigr] + \frac{1}{\epsilon}\EE_{\Y_t}\bigl[\bigl\|\B_1\inpone_t\bigr\|_{\Pmatone}^2\bigr]$ for $\eps > 0$. Choosing $\epsilon = 3\rho/(2(1-\rho))$ and utilizing Fact \ref{fact:factsonbounds} and the upper bound on the inputs $U_{\max}$, it follows that for all $t\ge\validfrom$,
		\begin{align*}
			\EE_{\Y_t} & \left[\norm{\estone_{t+1}}^2_{\Pmatone}\right] - (1-\tfrac{\rho}{2})\norm{\estone_{t}}^2_{\Pmatone} \le \frac{2(1-\rho)}{3\rho}\lambda_{\max}\bigl(\Pmatone\bigr)\norm{\B_1}^2 U_{\max}^2 n_1+\thehorror=:c.
		\end{align*}
		Therefore, for $t\ge\validfrom$, we have that
		\begin{align*}
			\EE_{\Y_{\kappa\kjumpvalidfrom}}  \left[\norm{\estone_{\kappa\kjumpvalidfrom+2}}^2_{\Pmatone}\right]& = \EE_{\Y_{\kappa\validfrom}}\left[\EE_{\Y_{\kappa\kjumpvalidfrom+1}}\left[\norm{\estone_{\kappa\kjumpvalidfrom+2}}^2_{\Pmatone}\right]\right] \\
			& \le \EE_{\Y_{\kappa\kjumpvalidfrom}}\left[\EE_{\Y_{\kappa\kjumpvalidfrom+1}}\left[(1-\tfrac{\rho}{2})\norm{\estone_{\kappa\kjumpvalidfrom+1}}^2_{\Pmatone} + c\right]\right]\\
			& \le (1-\tfrac{\rho}{2})^2 \norm{\estone_{\kappa\kjumpvalidfrom}}^2_{\Pmatone} + \bigl(1 + (1-\tfrac{\rho}{2})\bigr)c.
		\end{align*}
		Iterating the last inequality, it follows that
		\[
			\EE_{\Y_{\kappa\kjumpvalidfrom}}\left[\norm{\estone_{t+\kappa\kjumpvalidfrom}}^2_{\Pmatone} \right] \le (1-\tfrac{\rho}{2})^t \norm{\estone_{\kappa\kjumpvalidfrom}}^2_{\Pmatone} + \sum_{i=0}^{t-1}(1-\tfrac{\rho}{2})^ic,\qquad t\geq\validfrom,
		\]
       or
		\begin{align*}
			\EE_{\Y_{\kappa\kjumpvalidfrom}}&\left[\norm{\estone_{t+\kappa\kjumpvalidfrom}}^2_{\Pmatone} \right]  \le \norm{\estone_{\kappa\kjumpvalidfrom}}^2_{\Pmatone} +  (1-\tfrac{\rho}{2})^{-1}c, \qquad t\geq\validfrom.
		\end{align*}
		We can  conclude that there exists $\gamma^{(1)} = \gamma^{(1)}(\est_0) > 0$ such that
		\begin{equation}
		\label{e:estoneboundafter}
			\EE_{\Y_{\kappa\kjumpvalidfrom}}\bigl[\bigl\|\estone_t\bigr\|^2\bigr] \le \gamma^{(1)}\qquad \text{for all }t\ge\kappa\kjumpvalidfrom.
		\end{equation}
We can therefore conclude that
\[\EE_{\Y_{\kappa\kjumpvalidfrom}}\bigl[\bigl\|\est_t\bigr\|^2\bigr] \le \gamma^{(1)}+\gamma^{(2)}\qquad \text{for all }t\ge\kappa\kjumpvalidfrom.\]
Since the sequence $(\state_t)_{t\in\Nz}$ in \eqref{eq:system-a} is generated through the addition of independent mean-square bounded random variables and a bounded control input, and since $\kappa\kjumpvalidfrom < \infty$, it follows that there exists $\gamma > 0$ such that
		\[
			\EE_{\Y_{0}}\bigl[\norm{\state_t}^2\bigr] \le \gamma \qquad\text{for all }t\in\Nz,
		\]
		establishing the second claim \ref{maintheorem:msbdd} of Theorem \ref{thm:main}.
	\end{proof}
%		
%
%		Since $\kappa$ is finite and the sequence $(\est_t)_{t\in\Nz}$ is generated by linear system, standard arguments show that the bound above implies that there exists $\gamma^{(2)}_\kjumpvalidfrom \ge \gamma^{(2)}_\kappa$ such that
%		\begin{equation}
%		\label{e:esttwoboundafter}
%			\EE_{\Y_{\kappa\kjumpvalidfrom}}\bigl[\bigl\|\esttwo_{t}\bigr\|^2\bigr] \le \gamma^{(2)}_\kjumpvalidfrom\qquad \text{for all }t\ge\kappa\kjumpvalidfrom,
%		\end{equation}
%		From \eqref{e:estdecomposition}, \eqref{e:estoneboundafter} and \eqref{e:esttwoboundafter} we see that
%		\begin{equation}
%		\label{e:estboundafter}
%			\EE_{\Y_{\kappa\kjumpvalidfrom}}\bigl[\norm{\est_t}^2\bigr] \le \gamma^{(2)}\qquad \text{for all }t\ge\kappa\kjumpvalidfrom,
%		\end{equation}
%		where $\gamma^{(2)} \Let \gamma^{(1)} + \gamma^{(2)}_{\kjumpvalidfrom}$.
%
%
%
%		
%		 In conjunction with \eqref{e:estboundafter} this implies
%		\[
%			\EE_{\Y_{\kappa\kjumpvalidfrom}}\bigl[\norm{\state_t}^2\bigr] \le 2\bigl(\partialhorror + \gamma^{(2)}\bigr)\qquad \text{for all }t\ge\kappa\kjumpvalidfrom.
%		\]

\begin{proof}[Proof of Corollary \ref{cor:main}]
   The proof of Corollary \ref{cor:main} follows exactly the same reasoning as in the proof of Theorem \ref{thm:main}, up to the constraints in \eqref{eq:constraint-2} and \eqref{eq:constraint-3}. Also, rewriting the constraints \eqref{eq:statevar} and \eqref{eq:inputvar} as \eqref{eq:constraint-2} and \eqref{eq:constraint-3}, respectively, can be done similarly to the way we rewrote the cost in Theorem \eqref{thm:main}. It remains to show the upper bounds $\alpha^*$ and $\beta^*$ in Assumption \eqref{ass:stability2}.
   
   The constraint \eqref{eq:statevar} can be upper-bounded as follows:
   \begin{align*}
    &\EE_{\Y_t}\left[\State_t\transp \mathcal S\State_t +\mathcal L\transp \State_t\right]\\
    &\quad = \EE_{\Y_t}\left[\norm{\Aa \state_t+\Ba\Inp_t+\Da\Wnoise_t}_{\mathcal S}^2 +\mathcal L\transp (\Aa \state_t+\Ba\Inp_t+\Da\Wnoise_t)\right]\\
    &\quad \leq 3\EE_{\Y_t}\left[\norm{\Aa \state_t}^2_{\mathcal S}+\norm{\Ba\Inp_t}_{\mathcal S}^2+\norm{\Da\Wnoise_t}_{\mathcal S}^2\right]+\mathcal L\transp \Aa\hat\state_t+|\mathcal L\transp\Ba\Inp_t|\\
    &\quad \leq 3\tr{\Aa\transp \mathcal S\Aa\EE_{\Y_t}\left[\state_t\state_t\transp\right]+\Da\transp \mathcal S\Da\var{\wnoise}} + \mathcal L\transp \Aa\hat\state_t\\
    &\quad\quad 3Nm\sigma_{\max}(\Ba\transp \mathcal S\Ba)U_{\max}^2+ +\norm{\mathcal L\transp \Ba}_1U_{\max},
   \end{align*}
where the first inequality follows from the fact that $(a+b+c)^2\leq 3(a^2+b^2+c^2)$, for any $a,b,c>0$, and the noise being zero-mean, the second inequality follows from applying norm bounds between the 2 and $\infty$-norms and H\"older's inequality.
As for the constraint \eqref{eq:inputvar}, it can be upper-bounded as follows:
\begin{align*}
& \EE_{\Y_t}\left[\Inp_t\transp \tilde{\mathcal S}\Inp_t\right]\leq \sigma_{\max}(S)\norm{\Inp_t}^2\leq Nm\sigma_{\max}(S)\norm{\Inp_t}_\infty^2\leq Nm\sigma_{\max}(S)U_{\max}^2.
\end{align*}
This completes the proof.
\end{proof}

\section{Examples}\label{sec:examples}
%--------------------------------------------------------------------------------------
Consider the system \eqref{eq:system-a}-\eqref{eq:system-b} with the following matrices:
\[\A  = \smat{0.5 & 0 & 0 \\ 0 & 0 & -1\\ 0 & 1 & 0},
\B = \smat{1\\ 0 \\ 1}, \text{ and } C = I.\]
The simulation data was chosen to be $\state_0\sim\mathcal N(0,I)$, $w_t \sim \mathcal N(0,10I)$, $v_t = \mathcal N (0,10I)$, $\Wx = I$, $\Wu = 1$, $N=5$, $N_c=\kappa =2$, and $\varphi$ the usual piecewise linear saturation function with $\varphi_{\max}=1$. For this example the theoretical bound on the input is $U^*_{\max}=168.6783$ for a choice of $\epsilon = 10$.

We simulated the system for the discrete-time interval $[0,200]$ using Algorithm \ref{algo:general}, without the constraints \eqref{eq:statevar} and \eqref{eq:inputvar}. The coding was done using MATLAB and the optimization problem was solved using \texttt{yalmip} \citep{YALMIP} and \texttt{sdpt3} \citep{ref:sdpt3}. The computation of the matrices $\LambdaPhiE(P^\circ)$, $\LambdaPhi(P^\circ)$, $\LambdaWPhi(P^\circ)$, and $\LambdaPhiPhi(P^\circ)$ was done off-line using the steady state error covariance matrix $P^\circ$, as discussed in the previous section, via classical Monte Carlo integration \citep[Section 3.2]{RobertCasella-04} using $10^5$ samples.

The norms of the state trajectories for 200 different sample paths of the process and measurement noises are plotted in Figure \ref{fig:states} starting from $\state_0=\smat{97.38 & 100.19 & 99.78}^T$. The average state norm as well as the standard deviation of the state norm are depicted in Figure \ref{fig:statistics}, where it is clear that the proposed controller renders the system mean-square bounded. The average total cost normalized by time for this simulation is plotted in Figure \ref{fig:cost}.
\begin{figure}[h]
\centering
  \includegraphics[width=0.6\columnwidth]{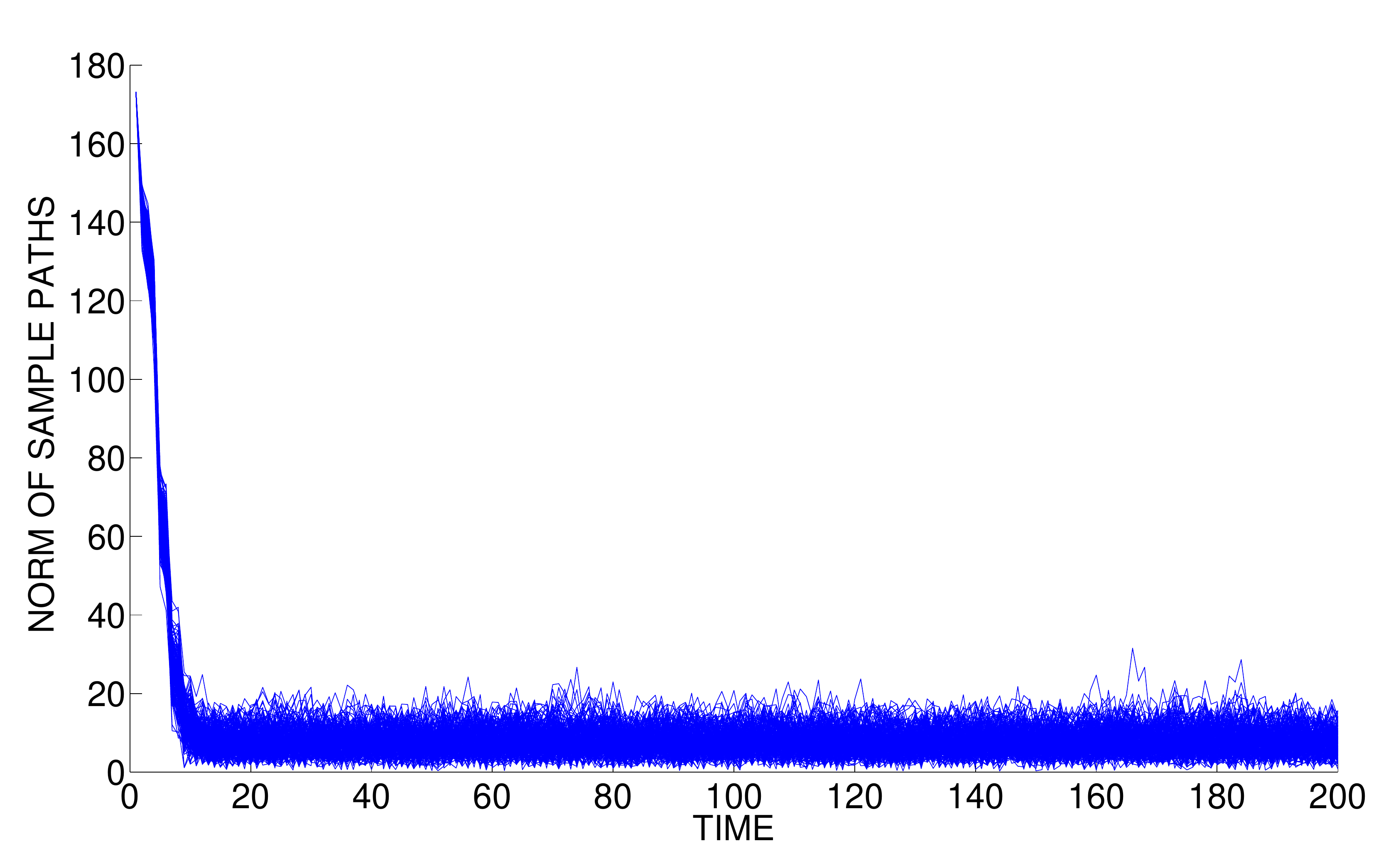}\\
  \caption{State norm for 200 realizations of the process and measurement noise sequences}\label{fig:states}
\end{figure}

\begin{figure}[h]
\centering
  \includegraphics[width=0.6\columnwidth]{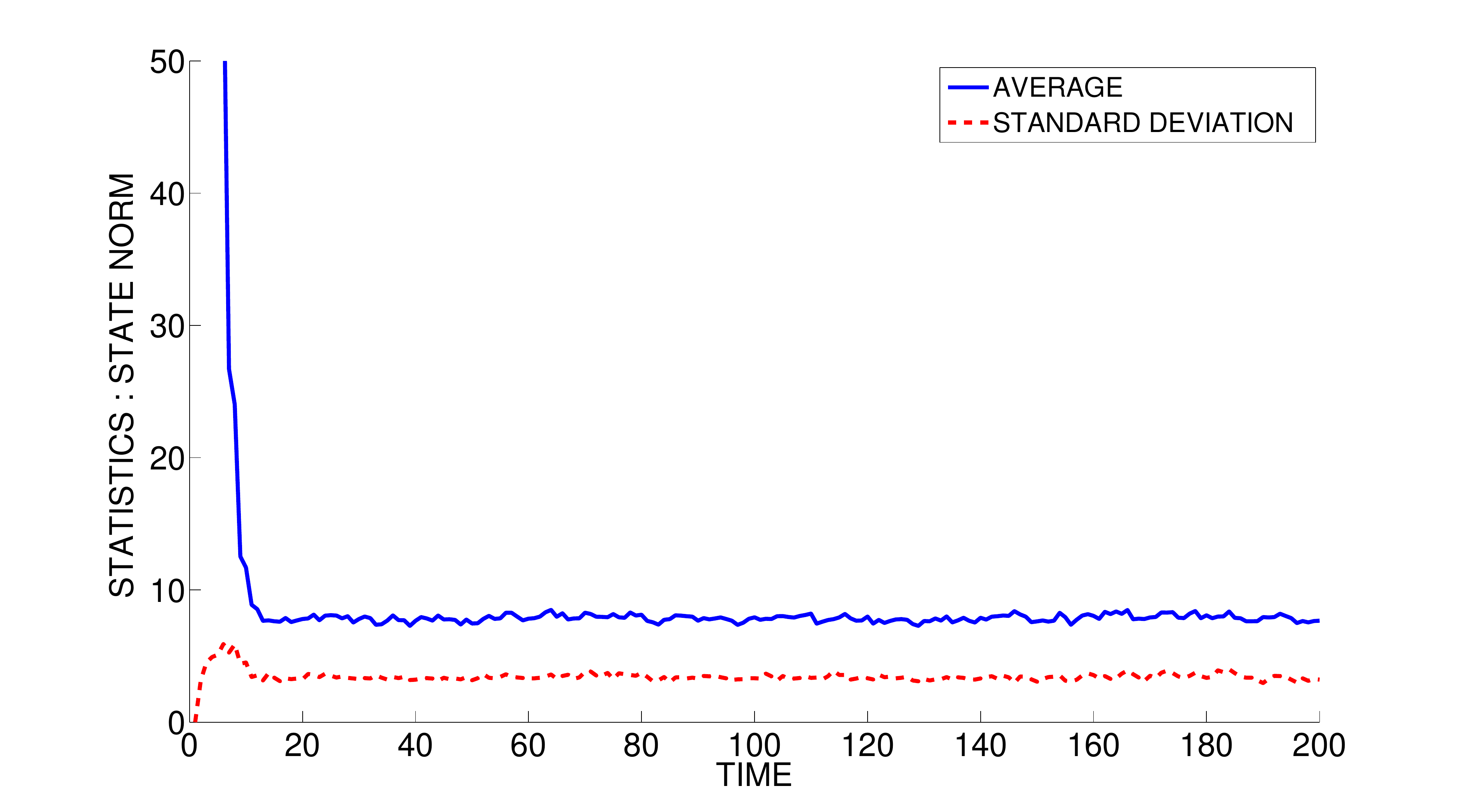}\\
  \caption{Average and standard deviation of the state norm for 200 realizations of the process and measurement noise sequences  }\label{fig:statistics}
\end{figure}

\begin{figure}[h]
\centering
  \includegraphics[width=0.6\columnwidth]{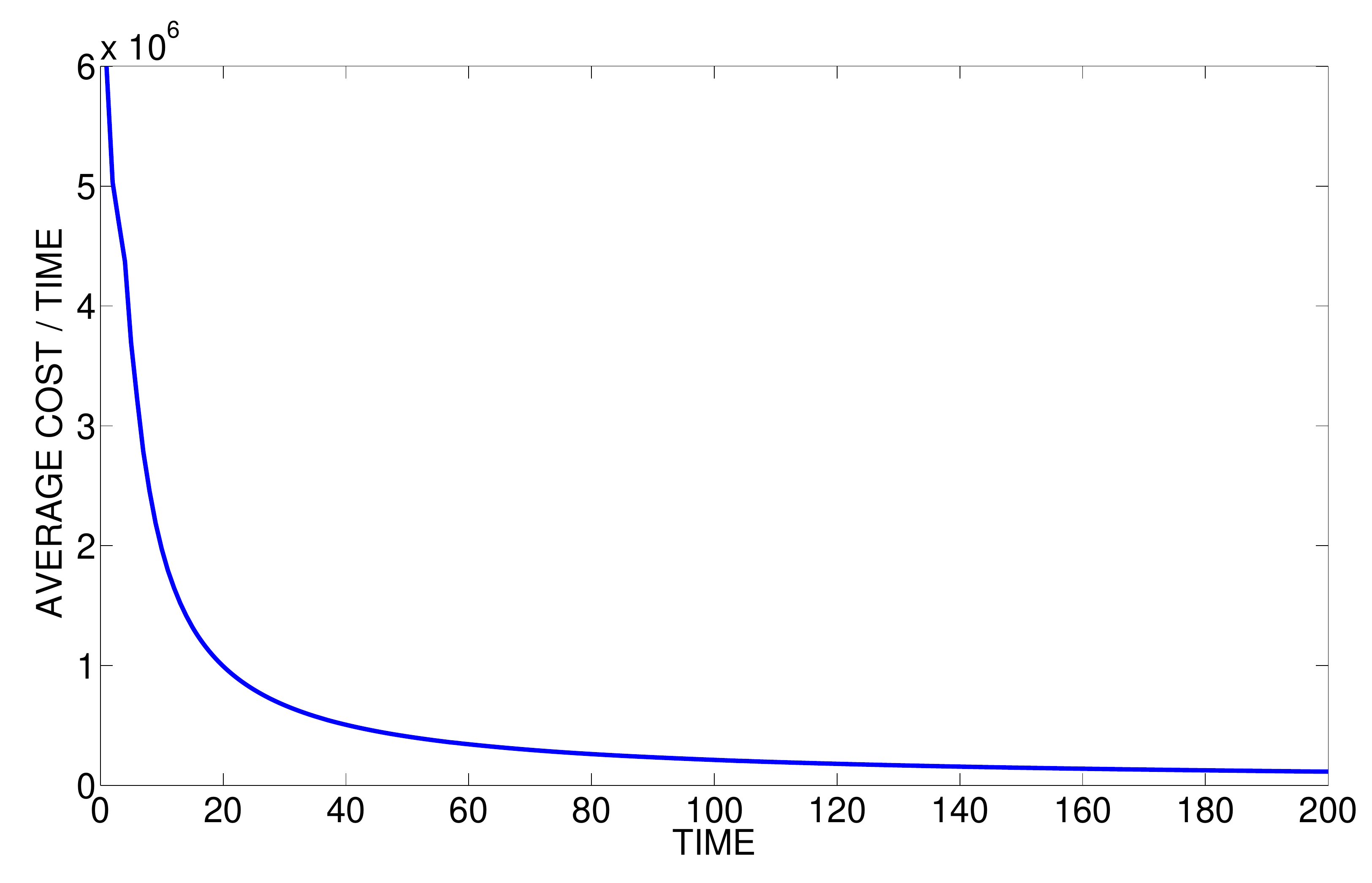}\\
  \caption{Total average cost}\label{fig:cost}
\end{figure}

%--------------------------------------------------------------------------------------
\section{Conclusions}\label{sec:conclusions}
%--------------------------------------------------------------------------------------
We presented a method for stochastic receding horizon control of discrete-time linear systems with process and measurement noise and bounded input policies. We showed that the optimization problem solved periodically is successively feasible and convex. Moreover, we illustrated how a certain stability condition can be utilized to show that the application of the receding horizon controller renders the state of the system mean-square bounded. We discussed how certain matrices in the cost function can be computed off-line and provided an example that illustrates our approach.

%-----------------------------------------------------------------------------
%\bibliographystyle{elsart-harv}
%\bibliography{../../ref}

\begin{thebibliography}{44}
\expandafter\ifx\csname natexlab\endcsname\relax\def\natexlab#1{#1}\fi
\expandafter\ifx\csname url\endcsname\relax
  \def\url#1{\texttt{#1}}\fi
\expandafter\ifx\csname urlprefix\endcsname\relax\def\urlprefix{URL }\fi

\bibitem[{Agarwal et~al.(2009)Agarwal, Cinquemani, Chatterjee, and
  Lygeros}]{AgarwalCinquemaniChatterjeeLygeros-09}
Agarwal, M., Cinquemani, E., Chatterjee, D., Lygeros, J., 2009. On convexity of
  stochastic optimization problems with constraints. In: European Control
  Conference. pp. 2827--2832.

\bibitem[{Anderson and Moore(1979)}]{ref:AndersonMooreFiltering}
Anderson, B., Moore, J., 1979. Optimal {F}iltering. Prentice-Hall.

\bibitem[{Batina(2004)}]{batinaPhDthesis}
Batina, I., 2004. Model predictive control for stochastic systems by randomized
  algorithms. Ph.D. thesis, Technische Universiteit Eindhoven.

\bibitem[{Bemporad and Morari(1999)}]{BemporadMorari-99}
Bemporad, A., Morari, M., 1999. Robust model predictive control: a survey.
  Robustness in Identification and Control 245, 207--226.

\bibitem[{Ben-Tal et~al.(2006)Ben-Tal, Boyd, and
  Nemirovski}]{Ben-TalBoydNemirovski-06}
Ben-Tal, A., Boyd, S., Nemirovski, A., 2006. Extending scope of robust
  optimization: Comprehensive robust counterparts of uncertain problems.
  Journal of Mathematical Programming 107, 63--89.

\bibitem[{Ben-Tal et~al.(2004)Ben-Tal, Goryashko, Guslitzer, and
  Nemirovski}]{ref:ben-tal04}
Ben-Tal, A., Goryashko, A., Guslitzer, E., Nemirovski, A., 2004. Adjustable
  robust solutions of uncertain linear programs. Mathematical Programming
  99~(2), 351--376.

\bibitem[{Bernstein(2009)}]{ref:Ber-09}
Bernstein, D.~S., 2009. Matrix {M}athematics, 2nd Edition. Princeton University
  Press.

\bibitem[{Bertsekas(2000)}]{ref:Ber-00Vol1}
Bertsekas, D.~P., 2000. Dynamic {P}rogramming and {O}ptimal {C}ontrol, 2nd
  Edition. Vol.~1. Athena Scientific.

\bibitem[{Bertsekas(2005)}]{ref:Ber-05}
Bertsekas, D.~P., 2005. Dynamic programming and suboptimal control: {a} survey
  from {ADP} to {MPC}. European Journal of Control 11~(4-5).

\bibitem[{Bertsekas(2007)}]{ref:Ber-07Vol2}
Bertsekas, D.~P., 2007. Dynamic {P}rogramming and {O}ptimal {C}ontrol, 3rd
  Edition. Vol.~2. Athena Scientific.

\bibitem[{Bertsimas and Brown(2007)}]{BertsimasBrown-07}
Bertsimas, D., Brown, D.~B., 2007. Constrained stochastic {LQC}: a tractable
  approach. IEEE Transactions on Automatic Control 52~(10), 1826--1841.

\bibitem[{Bertsimas et~al.(2009)Bertsimas, Iancu, and Parrilo}]{BerIanPar-09}
Bertsimas, D., Iancu, D.~A., Parrilo, P.~A., 2009. Optimality of affine
  policies in multi-stage robust optimizationPreprint available at:
  http://arxiv.org/abs/0904.3986.

\bibitem[{Blackmore and Williams(2007)}]{BlackmoreWilliams-07}
Blackmore, L., Williams, B.~C., July 2007. Optimal, robust predictive control
  of nonlinear systems under probabilistic uncertainty using particles. In:
  Proceedings of the American Control Conference. pp. 1759 -- 1761.

\bibitem[{Boyd and Vandenberghe(2004)}]{ref:BoyVan-04}
Boyd, S., Vandenberghe, L., 2004. Convex {O}ptimization. Cambridge University
  Press, Cambridge, sixth printing with corrections, 2008.

\bibitem[{Cannon et~al.(2009)Cannon, Kouvaritakis, and Wu}]{ref:CanKouWu-09}
Cannon, M., Kouvaritakis, B., Wu, X., July 2009. Probabilistic constrained
  {MPC} for systems with multiplicative and additive stochastic uncertainty.
  IEEE Transactions on Automatic Control 54~(7), 1626--1632.

\bibitem[{Chatterjee et~al.(2009)Chatterjee, Hokayem, and
  Lygeros}]{ref:ChaHokLyg-09}
Chatterjee, D., Hokayem, P., Lygeros, J., 2009. Stochastic receding horizon
  control with bounded control inputs: a vector space approach. IEEE
  Transactions on Automatic ControlUnder review.
  \url{http://arxiv.org/abs/0903.5444}.

\bibitem[{Ferrante et~al.(2002)Ferrante, Picci, and
  Pinzoni}]{FerrantePicciPinzoni-02}
Ferrante, A., Picci, G., Pinzoni, S., 2002. Silverman algorithm and the
  structure of discrete-time stochastic systems. Linear Algebra and its
  Applications 351/352, 219--242.

\bibitem[{Goulart et~al.(2006)Goulart, Kerrigan, and
  Maciejowski}]{GoulartKerriganMaciejowski-06}
Goulart, P.~J., Kerrigan, E.~C., Maciejowski, J.~M., 2006. Optimization over
  state feedback policies for robust control with constraints. Automatica
  42~(4), 523--533.

\bibitem[{Grant and Boyd(2000)}]{ref:boydCVX}
Grant, M., Boyd, S., December 2000. {CVX}: {M}atlab software for disciplined
  convex programming (web page and software).
  \url{http://stanford.edu/~boyd/cvx}.

\bibitem[{Hokayem et~al.(2009)Hokayem, Chatterjee, and
  Lygeros}]{ref:HokChaLyg-09}
Hokayem, P., Chatterjee, D., Lygeros, J., 2009. On stochastic model predictive
  control with bounded control inputs. In: Proceedings of the combined 48th
  IEEE Conference on Decision \& Control and 28th Chinese Control Conference.
  pp. 6359--6364, available at \url{http://arxiv.org/abs/0902.3944}.

\bibitem[{Horn and Johnson(1990)}]{ref:hornjohnson}
Horn, R.~A., Johnson, C.~R., 1990. Matrix {A}nalysis. Cambridge University
  Press, Cambridge.

\bibitem[{Jazwinski(1970)}]{ref:Jaz-70}
Jazwinski, A.~H., 1970. Stochastic {P}rocesses and {F}iltering {T}heory.
  Academic Press.

\bibitem[{Kamen and Su(1999)}]{KamenSu-99}
Kamen, E.~W., Su, J.~K., 1999. Introduction to Optimal Estimation. Springer,
  London, UK.

\bibitem[{{Klein Haneveld}(1983)}]{ref:Han-83}
{Klein Haneveld}, W.~K., 1983. On integrated chance constraints. In: Stochastic
  programming (Gargnano). Vol.~76 of Lecture Notes in Control and Inform. Sci.
  Springer, Berlin, pp. 194--209.

\bibitem[{{Klein Haneveld} and van~der Vlerk(2006)}]{ref:Han-06}
{Klein Haneveld}, W.~K., van~der Vlerk, M.~H., 2006. Integrated chance
  constraints: reduced forms and an algorithm. Computational Management Science
  3~(4), 245--269.

\bibitem[{Kumar and Varaiya(1986)}]{ref:KumVar-86}
Kumar, P.~R., Varaiya, P., 1986. Stochastic {S}ystems: {E}stimation,
  {I}dentification, and {A}daptive {C}ontrol. Prentice Hall.

\bibitem[{Lazar et~al.(2007)Lazar, Heemels, Bemporad, and
  Weiland}]{ref:lazarbemporad07}
Lazar, M., Heemels, W. P. M.~H., Bemporad, A., Weiland, S., 2007. Discrete-time
  non-smooth nonlinear {MPC}: stability and robustness. In: Lecture Notes in
  Control and Information Sciences. Vol. 358. Springer-Verlag, pp. 93--103.

\bibitem[{L{\"o}fberg(2003)}]{Loefberg-03}
L{\"o}fberg, J., 2003. Minimax {A}pproaches to {R}obust {M}odel {P}redictive
  {C}ontrol. Ph.D. thesis, Link{\"o}pings Universitet.

\bibitem[{L{\"o}fberg(2004)}]{YALMIP}
L{\"o}fberg, J., 2004. {YALMIP} : {A} {T}oolbox for {M}odeling and
  {O}ptimization in {MATLAB}. In: Proceedings of the CACSD Conference. Taipei,
  Taiwan.

\bibitem[{Maciejowski(2001)}]{ref:maciejowskibk}
Maciejowski, J.~M., 2001. Predictive {C}ontrol with {C}onstraints. Prentice
  Hall.

\bibitem[{Maciejowski et~al.(2005)Maciejowski, Lecchini, and
  Lygeros}]{MaciejowskiLecchiniLygeros-05}
Maciejowski, M., Lecchini, A., Lygeros, J., 2005. {NMPC} for complex stochastic
  systems using {Markov Chain Monte Carlo}. Vol. 358/2007 of Lecture Notes in
  Control and Information Sciences. Springer, Stuttgart, Germany, pp. 269--281.

\bibitem[{Mayne et~al.(2000)Mayne, Rawlings, Rao, and
  Scokaert}]{MayneRawlingsRaoScokaert-00}
Mayne, D.~Q., Rawlings, J.~B., Rao, C.~V., Scokaert, P. O.~M., Jun 2000.
  Constrained model predictive control: stability and optimality. Automatica
  36~(6), 789--814.

\bibitem[{Oldewurtel et~al.(2008)Oldewurtel, Jones, and
  Morari}]{OldewurtelJonesMorari-08}
Oldewurtel, F., Jones, C., Morari, M., Dec 2008. A tractable approximation of
  chance constrained stochastic {MPC} based on affine disturbance feedback. In:
  Conference on Decision and Control, CDC. Cancun, Mexico.

\bibitem[{Pemantle and Rosenthal(1999)}]{ref:PemRos-99}
Pemantle, R., Rosenthal, J.~S., 1999. Moment conditions for a sequence with
  negative drift to be uniformly bounded in {$L\sp r$}. Stochastic Processes
  and their Applications 82~(1), 143--155.

\bibitem[{Primbs(2007)}]{Primbs-07}
Primbs, J., 2007. A soft constraint approach to stochastic receding horizon
  control. In: Proceedings of the 46th IEEE Conference on Decision and Control.
  pp. 4797 -- 4802.

\bibitem[{Primbs and Sung(2009)}]{PrimbsSung-09}
Primbs, J.~A., Sung, C.~H., Feb. 2009. Stochastic receding horizon control of
  constrained linear systems with state and control multiplicative noise. IEEE
  Transactions on Automatic Control 54~(2), 221--230.

\bibitem[{Qin and Badgwell(2003)}]{ref:JoeQinBadgwell-03}
Qin, S.~J., Badgwell, T., Jul. 2003. A survey of industrial model predictive
  control technology. Control Engineering Practice 11~(7), 733--764.

\bibitem[{Ramponi et~al.(2009)Ramponi, Chatterjee, Milias-Argeitis, Hokayem,
  and Lygeros}]{ref:RamChaMilHokLyg-09}
Ramponi, F., Chatterjee, D., Milias-Argeitis, A., Hokayem, P., Lygeros, J.,
  2009. Attaining mean square boundedness of a marginally stable noisy linear
  system with a bounded control input. \url{http://arxiv.org/abs/0907.1436},
  submitted to IEEE Transactionson Automatic Control, Revised Jan 2010.

\bibitem[{Robert and Casella(2004)}]{RobertCasella-04}
Robert, C.~P., Casella, G., 2004. Monte {C}arlo {S}tatistical {M}ethods, 2nd
  Edition. Springer.

\bibitem[{Skaf and Boyd(2009)}]{SkafBoyd-09}
Skaf, J., Boyd, S., 2009. Design of affine controllers via convex optimization.
  \url{http://www.stanford.edu/~boyd/papers/affine_contr.html}.

\bibitem[{Toh et~al.(1999)Toh, Todd, and Tatuncu}]{ref:sdpt3}
Toh, K., Todd, M., Tatuncu, R., 1999. {SDPT3-- a Matlab software package for
  semidefinite programming}. Optimization Methods and Software~(11), 545--581,
  \url{http://www.math.nus.edu.sg/~mattohkc/sdpt3.html}.

\bibitem[{van Hessem and Bosgra(2003)}]{vanHessemBosgra-03}
van Hessem, D.~H., Bosgra, O.~H., 2003. A full solution to the constrained
  stochastic closed-loop {MPC} problem via state and innovations feedback and
  its receding horizon implementation. In: Proceedings of the 42nd IEEE
  Conference on Decision and Control. Vol.~1. pp. 929--934.

\bibitem[{Wang and Boyd(2009)}]{WangBoyd-09}
Wang, Y., Boyd, S., 2009. Peformance bounds for linear stochastic control.
  Systems and Control Letters 58~(3), 178--182.

\bibitem[{Yang et~al.(1997)Yang, Sontag, and
  Sussmann}]{ref:YangSontagSussmann-97}
Yang, Y.~D., Sontag, E.~D., Sussmann, H.~J., 1997. Global stabilization of
  linear discrete-time systems with bounded feedback. Systems and Control
  Letters 30~(5), 273--281.

\end{thebibliography}

%-----------------------------------------------------------------------------

%-----------------------------------------------------------------------------
\appendix
%-----------------------------------------------------------------------------

\section{Proof of Proposition \ref{prop:1}} \label{appendixA}

For $t=0$, $\state_0\sim \mathcal N(\est_{0|-1},\Pmat_{0|-1})$, with $\Pmat_{0|-1}>0$,
by assumption. Assume now that, for a given $t\geq 0$,
$f(\state_t|\Y_{t-1})$ is normal with mean $\est_{t|t-1}$ and
covariance matrix $\Pmat_{t|t-1}>0$. By Assumption \ref{ass:dynamics}-\ref{ass:distributions} and dynamics 
in \eqref{eq:system-b}, $f(\out_t|\state_t,\Y_{t-1})$ is also normal with mean $\C \state_t$ and
covariance matrix $\var{v}$. Hence, applying
Bayes rule  we may write
\[f(\state_t|\Y_t)=\frac{ f(\out_t|\state_t,\Y_{t-1}) f(\state_t|\Y_{t-1})}{f(\out_t|\Y_{t-1})},\]
where we have
$f(\out_t|\Y_{t-1})=\int  f(\out_t|\state_t,\Y_t)f(\state_t|\Y_{t-1}) \text{d} x_t$ by the Chapman-Kolmogorov equation. It follows that
\begin{align*}f(\state_t|\Y_t)&\propto f(\out_t|\state_t,\Y_{t-1})
f(\state_t|\Y_{t-1})\\
&\propto \exp\biggl( -\frac{1}{2}\Bigl[(\out_t-\C \state_t)\transp
\var{\vnoise}^{-1}(\out_t-\C \state_t)+(\state_t-\hat{\state}_{t|t-1})\transp
P_{t|t-1}^{-1}(\state_t-\hat{\state}_{t|t-1}) \Bigr]\biggr),
\end{align*}
where the proportionality constants do not depend on $\state_t$. Let us
now focus on the term within square brackets and write $\state$, $\out$,
$\hat{\state}$ and $P$ in place of $\state_t$, $\out_t$, $\hat{\state}_{t|t-1}$ and
$P_{t|t-1}$ for shortness. Expanding the products and collecting the
linear and quadratic terms in $\state$ one gets
\begin{align*}
& \state\transp (\C\transp \var{\vnoise}^{-1}\C+P^{-1})\state-2\state\transp (\C\transp \var{\vnoise}^{-1}\out+P^{-1}\hat{\state}) +(\out\transp\var{\vnoise}^{-1}\out+\hat{\state}\transp P^{-1}\hat{\state}).
\end{align*}
Since $P>0$ and $\var{\vnoise}>0$ by assumption, it follows that $\C\transp \var{\vnoise}^{-1}\C+P^{-1}>0$. If we let  $P_*=(\C\transp \var{\vnoise}^{-1}\C+P^{-1})^{-1}$, the expression
above can be rewritten as
\begin{align*}
& (P_*^{-1} \state)\transp  P_* (P_*^{-1} \state)-2 (P_*^{-1} \state)\transp P_* (\C\transp \var{\vnoise}^{-1}\out+P^{-1}\hat{\state})+(\out\transp\var{\vnoise}^{-1}\out+\hat{\state}\transp P^{-1}\hat{\state}).
\end{align*}
Completing the square, the latter expression becomes
\begin{eqnarray*}
[P_*^{-1} \state-(\C\transp \var{\vnoise}^{-1}\out+P^{-1}\hat{\state})]\transp  P_*
[P_*^{-1} \state -(\C\transp \var{\vnoise}^{-1}\out+P^{-1}\hat{\state})] + c,
\end{eqnarray*}
where $c$ depends on $\hat{\state}$ and $\out$ but not on $\state$. Factoring
$P_*^{-1}$ out of the two terms $P_*^{-1} \state-(\C\transp \var{\vnoise}^{-1}\out+P^{-1}\hat{\state})$ and simplifying yields
\begin{eqnarray*}
[ \state-P_*(\C\transp \var{\vnoise}^{-1}\out+P^{-1}\hat{\state})]\transp  P_*^{-1}
[ \state-P_*(\C\transp \var{\vnoise}^{-1}\out+P^{-1}\hat{\state})] + c.
\end{eqnarray*}
Hence,
\[f(\state_t|\Y_t)\propto \exp\left( -\frac{1}{2} (\state_t-\hat{\state}_*)\transp  P_*^{-1}
(x_t-\hat{x}_*)\right),\] 
with $\hat{x}_*=P_*(\C\transp \var{v}^{-1}y+P^{-1}\hat{x})$, where the missing normalization constant
does not depend on $x_t$. Hence, $f(x_t|\Y_t)$ is normal with mean
$\hat{x}_{t|t}=\hat{x}_*$ and covariance matrix $P_{t|t}=P_*$. Using
the matrix inversion lemma, $P_*=P-P\C\transp (\C P\C\transp +\var{v})^{-1}\C P$, which is~\eqref{eq:cupd}. To obtain~\eqref{eq:mupd},
replace $y$ by $(y-\C\hat{x})+\C\hat{x}$ in the expression
of $\hat{x}_*$ to get
\begin{align*}
\hat{x}_*&=P_*\left[\C\transp \var{v}^{-1}(y-\C\hat{x})+ \C\transp \var{v}^{-1}\C\hat{x} +P^{-1}\hat{x}\right] \\
&=P_*\C\transp \var{v}^{-1}(y-\C\hat{x}) + P_* (\C\transp \var{v}^{-1}\C+ P^{-1})\hat{x} \\
&=P_*\C\transp \var{v}^{-1}(y-\C\hat{x}) +\hat{x}.
\end{align*}
Finally, using the identity $I=(\C P \C\transp  +\var{v})(\C P \C\transp  +\var{v})^{-1}$ (where the inverse exists because $P\geq 0$ and $\var{v}>0$) and the
definition of $P_*$,
\begin{align*}
P_*\C\transp \var{v}^{-1}&= P_*\C\transp \var{v}^{-1}(\C P \C\transp  +\var{v})(\C
P\C\transp
+\var{v})^{-1}\\
&=P_*(\C\transp \var{v}^{-1}\C P \C\transp  +\C\transp )(\C P \C\transp
+\var{v})^{-1} \\
&=P_* (\C\transp \var{v}^{-1}\C +P^{-1})P\C\transp  (\C P \C\transp
+\var{v})^{-1} \\
&=P\C\transp  (\C P\bar  C\transp  +\var{v})^{-1},
\end{align*}
which leads to~\eqref{eq:mupd}.

To conclude the proof, we need to
compute the density $f(\state_{t+1}|\Y_{t})$ and prove~\eqref{eq:mpred}--\eqref{eq:cpred}.
Using the Chapman-Kolmogorov equation, we can compute the density as follows
\begin{equation}\label{eq:pred}
 f(\state_{t+1}|\Y_t)= \int f(\state_{t+1}|\state_t,\Y_t)f(\state_t|\Y_t)\text{d} \state_t.
\end{equation}
However, the explicit computation of
the integral in \eqref{eq:pred} is a very tedious exercise. We shall
instead rely on characteristic functions. Recall that the
characteristic function of an $n$-dimensional Gaussian random vector
$\xi$ with mean $\mu$ and covariance $\Sigma\geq 0$ is given by
$\Psi(z)=\EE[\exp{(\ii z\transp\xi)}]=\exp\left(\ii z\transp \mu-\frac{1}{2}z\transp \Sigma z\right)$, where $\ii\Let\sqrt{-1}$ and $z\in\RR^n$. The characteristic function of $\state_{t+1}$ given $\Y_t$ is then
\begin{align*}
\Psi(z)&=\EE[\exp(\ii z\transp \state_{t+1})|\Y_t] \\
&=\EE\left[\EE\big[\exp\big(\ii z\transp (\A \state_t+\B \inp_t+\wnoise_t) \big)\,\big|\,\state_t,\Y_t\big]\,\big|\,\Y_t\right] \\
&=\EE\left[\exp\big(\ii z\transp \A \state_t+ \ii z\transp \B g_t(\Y_t)\big)\right.\times \left.\EE\big[\exp(\ii z\transp \wnoise_t )\big|\state_t,\Y_t\big]\big|\Y_t\right],
\end{align*}
where we have used system dynamics in \eqref{eq:system-a} and the general definition of the feedback policies \eqref{e:policies}. Since $\wnoise_t$
is independent of $\state_t$ and $\Y_t$ and it is Gaussian with mean $0$
and covariance $\var{\wnoise}$, $\EE\big[\exp(\ii z\transp \wnoise_t)\big|x_t,\Y_t\big]=\exp\left(-\frac{1}{2}z\transp \var{\wnoise}z\right)$. Therefore the last expression in the chain becomes
\[
	\EE\left[\exp\left(\ii z\transp \A \state_t\right)\ |\ \Y_t\right]\exp\left(\ii z\transp \B g_t(\Y_t)\right)\exp\left(-\frac{1}{2}z\transp \var{\wnoise}z\right).
\]
It was proved above that, conditionally on $\Y_t$, $\state_t$ is Gaussian with mean $\est_{t|t}$ and covariance matrix $\Pmat_{t|t}$. Hence
\begin{align*}
    \EE[\exp(\ii z\transp \A \state_t)|\Y_t] & = \EE\left[\exp\big(\ii (\A\transp z)\transp  \state_t\big)\big|\Y_t\right]\\
    &=\exp\left(\ii (\A\transp z)\transp \est_{t|t}-\frac{1}{2}(\A\transp z)\transp \Pmat_{t|t}(\A\transp z)\right),
\end{align*}
and consequently
\begin{align*}
	\Psi(z)&=\exp\left(\ii z\transp \A\hat{\state}_{t|t}-\frac{1}{2}z\transp \A P_{t|t}A\transp z\right)\times \exp\left(\ii z\transp \B g_t(\Y_t)\right)\exp\left(-\frac{1}{2}z\transp \var{w}z\right)\\
	&= \exp\Bigl(\ii z\transp \big(\A\est_{t|t}+\B g_t(\Y_t)\big) -\frac{1}{2}z\transp (\A P_{t|t}\A\transp +\var{w})z\Bigr),
\end{align*}
which is the characteristic function of a Gaussian random vector
with mean $\hat{x}_{t+1|t}=\A\hat{x}_{t|t}+\B g_t(\Y_t)$ and
covariance matrix $P_{t+1|t}=\A P_{t|t}\A\transp +\var{\wnoise}>0$.

\section{ }\label{appendixB}

	\begin{lemma}\label{prop:Pbounds}
		Consider the system \eqref{eq:system-a}-\eqref{eq:system-b}, and suppose that $(\A, \Sigma_\wnoise^{1/2})$ be stabilizable and $(\A, \C)$ be observable. In addition, assume that $P_{0|0}\geq 0$. Then there exist constants $\rho>0$ and an integer $T$ large enough such that
        \begin{equation}
		      \label{eq:Pbound}
			  \EE_{\Y_t}\bigl[\norm{\state_t - \hat\state_t}^2\bigr] = \tr{P_{t|t}} \leq \rho \qquad \text{for all }t\ge T.
		\end{equation}
	\end{lemma}
	\begin{proof}
		First, observe that
		\[
			\sum_{i=0}^{\kappa_1-1} \A^{i}\Sigma_{\wnoise}(\A^{i})\transp =
			\Bigl[\Sigma_{\wnoise}^{1/2}	\;	\A \Sigma_{\wnoise}^{1/2}	\;	\cdots\;	\A^{\kappa_1-1} \Sigma_{\wnoise}^{1/2}\Bigr]
			\Bigl[\Sigma_{\wnoise}^{1/2}	\;	\A \Sigma_{\wnoise}^{1/2}	\;	\cdots\;	\A^{\kappa_1-1} \Sigma_{\wnoise}^{1/2}\Bigr]\transp,
		\]
		and since $(\A, \Sigma_{\wnoise}^{1/2})$ is controllable by Assumption \ref{ass:stability1}-\ref{ass:stability1:systemandnoise}, we see that there exists $\kappa_1\in\NN$ such that for all $k \ge \kappa_1$ the rank of $\Bigl[\Sigma_{\wnoise}^{1/2}	\;	\A \Sigma_{\wnoise}^{1/2}	\;	 \cdots\;	 \A^{\kappa_1-1} \Sigma_{\wnoise}^{1/2}\Bigr] = n$; indeed, $\kappa_1$ is the reachability index of $(\A, \Sigma_{\wnoise}^{1/2})$. Thus, $\sum_{i=0}^{\kappa_1-1} \A^{i}\Sigma_{\wnoise}(\A^{i})\transp$ is positive definite, and therefore, there exists some $\delta_1', \delta_2' > 0$ such that
\begin{equation}\label{eq:delta1bound}
\delta_1' I \le \sum_{i=0}^{\kappa_1-1} \A^{i}\Sigma_{\wnoise}(\A^{i})\transp \le \delta_2' I.
\end{equation}
Second, observe that\footnote{Here $\otimes$ denotes the standard Kronecker product.}
		\[
			\sum_{i=0}^{\kappa_2 -1} (\A^{i})\transp\C\transp\Sigma_{\vnoise}^{-1}\C\A^{i} =
			\begin{bmatrix}
				\C\\
				\C\A\\
				\vdots\\
				\C\A^{\kappa_2 -1}
			\end{bmatrix}\transp
			%\begin{bmatrix}
			%	\Sigma_{\vnoise}^{-1}	&	0						&	\cdots	&	0\\
			%	0						&	\Sigma_{\vnoise}^{-1}	&	\cdots	&	0\\
			%	0						&	0						&	\ddots	&	0\\
			%	0						&	0						&	\ddots	&	\Sigma_{\vnoise}^{-1}
			%\end{bmatrix}
			\bigl(I_{\kappa_2 }\otimes\Sigma_{\vnoise}^{-1}\bigr)
			\begin{bmatrix}
				\C\\
				\C\A\\
				\vdots\\
				\C\A^{\kappa_2 -1}
			\end{bmatrix}.
		\]
		Since $(\A, \C)$ is observable by assumption, there exists $\kappa_2 \in\NN$ such that the rank of the matrix $\Bigl[\C\transp\;\A\transp\C\transp\;\cdots\;(\A^{\kappa_2 -1})\transp\C\transp\Bigr]\transp$ is $n$. The matrix $I_{\kappa_2 }\otimes \Sigma_{\vnoise}^{-1}$ is clearly positive definite by Assumption \ref{ass:stability1}'\ref{ass:stability1:systemandnoise}, and therefore, we see that there exists $\delta_1'', \delta_2'' > 0$ such that

		\begin{equation}\label{eq:delta2bound}
			\delta_1'' I \le \sum_{i=0}^{\kappa_2 -1} (\A^{i})\transp\C\transp\Sigma_{\vnoise}^{-1}\C\A^{i} \le \delta_2'' I.
		\end{equation}
		Third, the conditions of Lemma 7.1 in \cite[pp.\ 234]{ref:Jaz-70} are satisfied, as \eqref{eq:delta1bound} and \eqref{eq:delta2bound} hold with $\delta_1 = \min\{\delta_1', \delta_1''\}$ and $\delta_2 = \max\{\delta_2', \delta_2''\}$, and the bound $P_{t|t}\le \rho' I$ for some $\rho' > 0$ is established for all $t\ge T\Let \max\{\kappa_1,\ \kappa_2\}$. The assertion now follows immediately from:
		\[
			\EE_{\Y_t}\bigl[\norm{\state_t - \hat\state_t}^2\bigr] = \tr{P_{t|t}}\leq n\lambda_{\max}(P_{t|t})\leq n\rho' \teL \rho.\qedhere
		\]
	\end{proof}

\end{document}